\documentclass{amsart}
\usepackage{amsmath, amssymb, amsthm}
\usepackage{mathtools}
\usepackage{amscd}
\usepackage{graphicx} 
\usepackage{todonotes}
\usepackage{tabularx}
\usepackage{nicefrac}
\usepackage{framed}
\usepackage{xcolor}
\colorlet{shadecolor}{green!25}
\usepackage{hyperref}
\usepackage{url}
\usepackage{enumitem}

\newtheorem{thm}{Theorem}[section]

\newtheorem{prop}[thm]{Proposition}
\newtheorem{lemma}[thm]{Lemma}

\newtheorem*{thm*}{Theorem}
\newtheorem*{alg*}{Algorithm}
\newtheorem*{lemma*}{Lemma}

\theoremstyle{remark}
\newtheorem{rmk}[thm]{Remark}
\newtheorem*{rmk*}{Remark}

\newtheorem*{notation*}{Notation}

\newtheorem*{example*}{Example}

\theoremstyle{definition}

\newtheorem*{defn*}{Definition}

\newcommand{\mybf}{\mathbb}

\newcommand{\bC}{\mybf{C}}

\newcommand{\bZ}{\mybf{Z}}
\newcommand{\bQ}{\mybf{Q}}

\newcommand{\Q}{\mybf{Q}}
\newcommand{\Z}{\mybf{Z}}
\newcommand{\C}{\mybf{C}}

\newcommand{\cO}{\mathcal{O}}

\newcommand{\al}{\alpha}
\newcommand{\QQ}{\mathbb{Q}}
\newcommand{\Qbar}{\overline{\mathbb{Q}}}

\providecommand{\abs}[1]{\left\lvert#1\right\rvert}

\newcommand{\ep}{\epsilon}
\newcommand{\ra}{\rightarrow}

\newcommand{\consalph}[1]{\alpha_{\sigma}\left({#1}\right)}
\newcommand{\consbet}[1]{\beta_{\sigma}\left({#1}\right)}

\DeclareMathOperator{\id}{id}
\DeclareMathOperator{\Gal}{Gal}

\title{Wandering points for the Mahler measure}
\author[Fili]{Paul Fili}
\email{paul.fili@okstate.edu}
\address{Oklahoma State University, Stillwater, OK, 74078, USA}

\author[Pottmeyer]{Lukas Pottmeyer}
\email{lukas.pottmeyer@uni-due.de}
\address{Universit\"{a}t Duisberg-Essen, 45117 Essen, Germany}

\author[Zhang]{Mingming Zhang}
\email{mingmiz@okstate.edu}
\address{Oklahoma State University, Stillwater, OK, 74078, USA}

\date{\today}

\begin{document}

\begin{abstract}
Mahler's measure defines a dynamical system on the algebraic numbers. In this paper, we study the problem of which number fields have points which wander under the iteration of Mahler's measure. We completely solve the problem for all abelian number fields, and more generally, for all extensions of the rationals of degree at most five.
\end{abstract}

\dedicatory{Dedicated to the memory of Andrzej Schinzel, 1937--2021}

\maketitle


\section{Introduction}
Fix an algebraic closure $\Qbar\subseteq \C$ of $\Q$ and let $\alpha\in\Qbar$ be an algebraic number with minimal polynomial \[f(x) = a_n x^n + \cdots + a_1 x + a_0\in\bZ[x].\] Suppose that $f(x)$ factors over $\bC$ as $f(x) = a_n(x-\alpha_1)\cdots (x-\alpha_n)$. Then the \emph{Mahler measure} of $\alpha$, denoted $M(\alpha)$, is defined to be the usual Mahler measure of its minimal polynomial $f(x)$:
\[
 M(\alpha) = \int_{0}^1 \log\,\abs{f(e^{2\pi i t})}\,dt = \abs{a_n}\prod_{i=1}^n \max\{1,\abs{\alpha_i}\}.
\]
The Mahler measure of $\alpha$ satisfies that $M(\alpha) \geq 1$ for all $\alpha\in\Qbar$, is Galois-invariant, satisfies $M(\alpha)=M(\alpha^{-1})$ for all $\alpha \in \overline{\Q}^*$, and $M(\alpha) = 1$ if and only if $\al$ is zero or a root of unity. Further, it is directly related to its absolute, logarithmic Weil height $h(\alpha)$ via the formula $[\bQ(\alpha):\bQ]\cdot h(\alpha) = \log M(\alpha)$. Perhaps the most famous open question regarding the Mahler measure is \emph{Lehmer's problem}, posed by D.H. Lehmer in 1933 \cite{Lehmer}, which asks if one can find, for any $\ep>0$, an algebraic number $\alpha\in \Qbar$ such that $1<M(\alpha)<1+\ep$. The smallest known Mahler measure greater than $1$ was found by Lehmer himself in his 1933 paper, and is the Salem number $\tau>1$ which is a root of the polynomial $x^{10}+x^9-x^7-x^6-x^5-x^4-x^3+x+1$. It has Mahler measure $M(\tau)= \tau = 1.17\ldots$ As no one has found an algebraic number of smaller Mahler measure strictly greater than $1$, the conjecture that the values of the Mahler measure are bounded away from $1$ has become known as the \emph{Lehmer conjecture}.\footnote{Somewhat erroneously, as Lehmer himself never conjectured this, merely asking if values closer to $1$ could be found.}

It is easy to see from the formula for the Mahler measure above that $M(\al)$ is an algebraic number for any $\al\in\Qbar$ (in fact, as we will discuss in more detail below, it is a particular class of algebraic integer called a \emph{Perron number}, a concept introduced by Lind \cite{Lind83}). It follows that $M : \Qbar \ra \Qbar$ so the Mahler measure can be thought of as a dynamical system on the algebraic numbers. As is common in dynamical systems, we write $M^0(\al) = \al$ and let $M^n(\al) = M(M^{n-1}(\al))$ for all $n\geq 1$. Dubickas \cite[\S 4]{DubickasClose} appears to have been the first to study the Mahler measure as a dynamical system. The question of which algebraic numbers could be Mahler measures has been studied quite extensively; see e.g. \cite{AdlerMarcus,Lind83,Lind84,BoydPerron,DubickasNrsMM,DixonDubickas,DubickasSalemNonrecip,Schinzel17+1}, and the introduction to \cite{FPZ2020} for more background. From a dynamical point of view, this question is related to the study of backward orbits of a given algebraic number. Many questions remain open in this direction. For instance, Schinzel noted in \cite{Schinzel17+1} that it is not known whether $1+\sqrt{17}$ is the Mahler measure of some algebraic number. This question is still open.

As usual in studying dynamical systems, we denote the (forward) orbit of $\al$ under $M$ by
\[
 \cO_M(\al) = \{ M^n(\al) : n\geq 0 \},
\] 
and we are interested in determining when the orbit is finite, in which case we say $\al$ is \emph{preperiodic}, or when it is infinite, in which case we say $\al$ is a \emph{wandering point}. Dubickas observed that, for algebraic integers $\al\geq 1$, $M$ is a nondecreasing function, and so 
\[
 M(\al) \leq M^2(\al)\leq M^3(\al)\leq \cdots 
\]
and this implies that the Mahler measure either keeps increasing, in which case the orbit of $\al$ under $M$ is infinite and $\alpha$ is a wandering point for the Mahler measure, or else the orbit is finite, and ends at a fixed point of the Mahler measure, that is, at an algebraic integer $\beta\geq 1 $ such that $M(\beta) = \beta$. (In particular, there are no points of strict period $n>1$ for the Mahler measure.)

One of the most interesting observations regarding the iteration of Mahler measure which Dubickas made is that the fixed points for the Mahler measure must be either natural numbers, Pisot numbers, or Salem numbers (the fixed points of the related dynamical system induced by the multiplicative height, have recently been classified by Dill \cite{Dill}). As the Mahler measure of Pisot numbers is well-known to be bounded away from $1$ by Smyth's celebrated theorem \cite{Smy}, it is a folklore conjecture that Lehmer's conjecture is equivalent to the question of whether the set of Salem numbers is bounded away from $1$. For this reason, it seems particularly interesting that the `nontrivial' fixed points of Mahler's measure (that is, those which are fixed, and have Mahler measure not already known to be bounded away from $1$) consist of only Salem numbers, and suggests that a better understanding of the iteration of Mahler's measure may yield some clues towards the Lehmer problem.

After Dubickas, the first nontrivial results about the iteration of Mahler's measure were by the third author \cite{ZhangThesis}, who proved among other things that every $\alpha$ whose minimal polynomial has degree $3$ or smaller must have a finite orbit, but that in degree $4$, it was possible to find wandering points. In a previous paper \cite{FPZ2020}, the authors proved several results about the existence of algebraic integers  with arbitrarily long, but finite, orbits of the Mahler measure and the behavior of units when the Galois group is large.

\begin{table}\label{tab1}
\begin{center}
\begin{tabular}{c||c|c|c|c|c}
$K$ & $S_4$ & $A_4$ & $D_4$ & $C_4$ & $C_2 \times C_2$ \\ \hline\hline
totally imaginary & $P$ & $P$ & $P$ & $P$ & $P$ \\
signature $(2,1)$ & $W$ & $W$ & $P$ & -- & -- \\
totally real & $W$ & $W$ & $W$ & $W$ & $P$ \\ 
\end{tabular}\caption{Classification according to the signature and the Galois group of the Galois closure of quartic number fields $K$, whether there are wandering units (W) or whether all elements in $K$ are preperiodic (P).}
\end{center}
\end{table}

Dubickas then posed to the authors (private communication) the question of whether one could classify which number fields have wandering points and which do not. Clearly, from the third author's work \cite{ZhangThesis}, fields with wandering points must have degree at least $4$. However, not all such fields have wandering points, as we will show. In this paper, we give some results toward this question of Dubickas, completely solving the problem in the abelian case. Table \ref{tab1} summarizes our results on quartic fields, giving a classification of quartic number fields with wandering points.  For fields of degree $5$ over $\bQ$, we prove that one can always find a wandering point:
\begin{thm}
Let $K/\Q$ be an extension of degree five. Then there exists an algebraic unit in $K$ which is wandering under iteration of the Mahler measure $M$.
\end{thm}

The main theorem of this paper is the following, which completely classifies the existence of wandering points for the Mahler measure in abelian fields:
\begin{thm}\label{thm:main-abelian}
Let $K/\Q$ be an abelian number field. Then all elements in $K$ are preperiodic under iteration of $M$ if and only if the maximal totally real subfield of $K$ has Galois group isomorphic to $C_1$, $C_2$, $C_3$, or $C_2\times C_2$, where $C_n$ denotes the cyclic group of order $n$. In all other cases, $K$ contains a wandering algebraic unit.
\end{thm}
Note that this implies that in imaginary abelian extensions with Galois group $C_2^3$, there are no wandering points. We could not find a field of degree larger than $8$ with this property. In the course of proving our theorem, we also classify all Galois extensions of $\Q$ of degree at most $9$ which contain a wandering point under iteration of $M$. In particular, by Theorem \ref{thm:main-abelian} any Galois extension of prime degree $p\geq 5$ contains a wandering algebraic unit.

We note that, in all cases in which we prove the existence of wandering points, we actually prove the existence of wandering \emph{units}. It remains open whether there are fields which contain a wandering point, but not a wandering unit.

Portions of this paper previously appeared in the third author's Ph.D. thesis \cite{ZhangPhDThesis}. The authors would like to thank Art\={u}ras Dubickas for posing the question which led to this paper, and Gabriel Dill for useful comments on an earlier version of this paper.

\section{Preliminaries}
When discussing Galois groups of number fields, we will use the usual notation $S_n$ and $A_n$ for the symmetric and alternating group on $n$ letters, respectively, $D_n$ for the dihedral group on $n$ letters, and $C_n$ for the cyclic group of order $n$. We will also have occasion to use two less common groups: we will denote by $F_5$ the \emph{Frobenius group} on $5$ letters, which is a group of order $20$ given by the generators $\sigma,\tau$ under the relations $\sigma^5=\tau^4=1$ and $\tau\sigma\tau^{-1}= \sigma^2$. (Indeed, any choice of $\sigma,\tau\in F_5$ with $\sigma$ of order $5$ and $\tau$ of order $4$, satisfies either $\tau\sigma\tau^{-1}=\sigma^2$ or $\tau^{-1} \sigma \tau = \sigma^2$.) We also denote by $Q_8$ the \emph{quaternion group} which has generators $\sigma,\tau$ satisfying the relations $\sigma^4=1$, $\sigma^2=\tau^2$, and $\tau\sigma\tau^{-1}=\sigma^3$. For more information about the classification of transitive permutation groups on $n$ letters, see for example \cite{ConwayTransitive}. One can also obtain lists of all transitive subgroups of $S_n$ via databases such as the LMFDB \cite{lmfdb} or via computer software packages such as Sage \cite{sagemath}.

We start with a lemma which will be frequently used to prove that an algebraic number $\al$ is a wandering point for the Mahler measure. Recall that an algebraic number $\alpha$ is called \emph{torsion-free} if, for all $\sigma\in\Gal(\overline{\Q}/\Q)$ such that $\sigma(\al)\neq \al$, the algebraic number $\alpha/\sigma(\alpha)$ is not a root of unity. (This terminology seems to have been introduced by Dubickas \cite{DubickasDegreeLinForm}.) Torsion-free numbers are interesting in part because $\alpha$ being torsion-free is equivalent to the condition that, for any $n\in \mathbb{N}$, we have $[\Q(\alpha^n):\Q]=[\Q(\alpha):\Q]$. 

\begin{lemma}\label{lem:torsionfree}
If $\alpha$ is torsion-free, $\alpha\neq \pm1$, and $M^k(\al)=\al^n$ for some positive integers $k$ and $n$ with $n\geq 2$, then $\al$ is a wandering point for $M$. In particular, if for any $\alpha \in \overline{\Q}^*$ not a root of unity we have $M^k(\alpha)=M^\ell(\alpha)^n$ for positive integers $k,\ell,n$, with $k>\ell$ and $n\geq 2$, then $\alpha$ is a wandering point under iteration of $M$.
\end{lemma}
\begin{proof}
We start by proving the first statement.  
Given that $\alpha$ is torsion-free we know $[\mathbb{Q}(\alpha^n):\mathbb{Q}]=[\mathbb{Q}(\alpha):\mathbb{Q}]$. Let the set of conjugates of $\alpha$ be $\{\alpha_1, ...,\alpha_m\}$, then the set of conjugates of $\alpha^n$ is $\{\alpha_1^{n}, ...,\alpha_m^{n}\}$, hence 
\[
M(\alpha^n)=\underset{\left|\alpha_{i}^{n}\right|\geq 1}{\prod}\left|\alpha_{i}^{n}\right|= \underset{\left|\alpha_{i}\right|\geq 1}{\prod}\left|\alpha_{i}\right|^{n}=\bigg(\underset{\left|\alpha_{i}\right|\geq 1}{\prod}\left|\alpha_{i}\right|\bigg)^{n}=(M(\alpha))^n.
\]
Since the Mahler measure of any algebraic number is always a Perron number \cite{DixonDubickas}, it is in particular always torsion-free. Hence, we get
\[
M^{2k}(\alpha)=M^k(\alpha^n)=M^k(\alpha)^n=\alpha^{n^2},
\]
and inductively it follows $M^{rk}(\alpha)=\alpha^{n^r}$ for all $r\in\mathbb{N}$. This implies that $\alpha$ is a wandering point under iteration of $M$.

Now let $\alpha$ be an arbitrary algebraic number and let $k>\ell\geq 1$ be integers, with $M^k(\alpha)=M^\ell(\alpha)^n$ for some $n\geq 2$. Then we have $M^{k-\ell}(M^\ell(\alpha))=M^\ell(\alpha)^n$. As noted above $M^\ell(\alpha)$ is torsion-free. Hence, it follows that $M^\ell(\alpha)$ is wandering under iteration of $M$. Thus, $\alpha$ is a wandering point as well.
\end{proof}

\begin{lemma}\label{lem:deg5}
Let $K$ be a number field with an odd number of pairwise not complex conjugate embeddings $\sigma_{1},\ldots, \sigma_{2n+1}$ into $\C$. Assume that for some $q\leq n$ the $2q$ embeddings $\sigma_{n-q+2},\ldots, \sigma_{n+q+1}$ are complex, and the other embeddings are real. Then there is an algebraic unit $\alpha\in K$ satisfying
\begin{enumerate}[label=(\roman*)]
\item $\abs{\sigma_1(\alpha)} > \ldots > \abs{\sigma_{n+1}(\alpha)} > 1 > \abs{\sigma_{n+2}(\alpha)} > \ldots > \abs{\sigma_{2n+1}(\alpha)}$, and
\item $\abs{\sigma_1(\alpha)\sigma_{n+2}(\alpha)}<1$.
\end{enumerate}
\end{lemma}
\begin{proof}
Let $\mathcal{O}_K^*$ be the group of algebraic units in $K$. Set 
\[
\varepsilon_i = \begin{cases}
1 & \text{ if } \sigma_i \text{ is real} \\ 2 & \text{ if } \sigma_i \text{ is complex}
\end{cases}
\]
and define for all $\alpha \in \mathcal{O}_K^*$
\[
L(\alpha)=(\varepsilon_i \log \vert \sigma_i(\alpha)\vert)_{i\in\{1,\ldots 2n+1\}}.
\]
By Dirichlet's unit theorem, $L(\mathcal{O}_K^*)$ is a lattice of rank $2n$ in the hyperplane 
\[
\bigg\{(x_1,\ldots,x_{2n}, -\sum_{i=1}^{2n} x_i)\vert x_1,\ldots,x_{2n}\in\mathbb{R}\bigg\} \subseteq \mathbb{R}^{2n+1}.
\]
We fix a fundamental domain of this lattice, and some $B\in\mathbb{R}$ which is greater than any vector spanning this domain. Let $k \geq n+1$ be a large integer, which we will fix later on. Then, there is an element $(x_1,\ldots,x_{2n},y)\in L(\mathcal{O}_K^*)$ such that
\[
x_i \in \begin{cases}
((k-i)B,(k-i+1)B) & \text{ for } i\in\{1,\ldots,n-q+1\} \\
(2(k-i)B,2(k-i+1)B) & \text{ for } i \in \{n-q+2,\ldots,n+1\} \\
(-2(i-n+k-1)B,-2(i-n+k-2)B) & \text{ for } i \in \{n+2,\ldots, n+q+1\} \\
(-(i-n+k-1)B,-(i-n+k-2)B) & \text{ for } i \in \{n+q+2,\ldots,2n\}.
\end{cases}
\]
In particular, there is $\alpha\in\mathcal{O}_K^*$ with $L(\alpha)=(x_1,\ldots,x_n,y)$. By construction, this $\alpha$ satisfies
\[
\abs{\sigma_1(\alpha)} > \ldots > \abs{\sigma_{n+1}(\alpha)} > 1 > \abs{\sigma_{n+2}(\alpha)} > \ldots > \abs{\sigma_{2n}(\alpha)},
\]
and
\[
\abs{\sigma_1(\alpha)\sigma_{n+2}(\alpha)}<1.
\]
We are left to prove that $\abs{\sigma_{2n+1}(\alpha)}<\abs{\sigma_{2n}(\alpha)}$. In the following $c_1,c_2,c_3$ denote constants only depending on $n$ and $q$ (in particular they are independent on $k$).

Let us first assume that $q<n-1$. Then $\sigma_{2n}$ and $\sigma_{2n+1}$ are real embeddings, and we have
\begin{align}\label{eq:Dirichlet1}
\frac{1}{B}\sum_{i=1}^{2n} x_i &> \sum_{i=1}^{n-q+1} (k-i) + \sum_{i=n-q+2}^{n+1}2(k-i) \nonumber\\ &- \sum_{i=n+2}^{n+q+1} 2(i-n+k-1) - \sum_{i=n+q+2}^{2n} (i-n+k-1)\nonumber  \\
 &=(n-q+1)k + 2qk -2qk - (n-q-1)k + c_1 = 2k+c_1.
\end{align}
Therefore, for $k$ large enough, we have $y=-\sum_{i=1}^{2n}x_i < B(-k-n) < x_{2n}$. This implies $\abs{\sigma_{2n+1}(\alpha)}<\abs{\sigma_{2n}(\alpha)}$, which proves the lemma in the case $q<n-1$.

If $q=n-1$, then $\sigma_{2n}$ is complex and $\sigma_{2n+1}$ is real. Morover, we have
\begin{align}
\frac{1}{B}\sum_{i=1}^{2n} x_i &> \sum_{i=1}^{2} (k-i) + \sum_{i=3}^{n+1}2(k-i) - \sum_{i=n+2}^{2n} 2(i-n+k-1) \nonumber \\
& =2k+2(n-1)k-2(n-1)k+c_2 = 2k +c_2.
\end{align}
Hence, for $k$ large enough we have $y=-\sum_{i=1}^{2n}x_i < B(-k-n) < \frac{1}{2} x_{2n}$. This implies $\abs{\sigma_{2n+1}(\alpha)}<\abs{\sigma_{2n}(\alpha)}$.

Lastly, we concider the case $q=n$. Then $\sigma_{2n}$ and $\sigma_{2n+1}$ are both complex, and
\begin{align}
\frac{1}{B}\sum_{i=1}^{2n} x_i &>  (k-1) + \sum_{i=2}^{n+1} 2(k-i) - \sum_{i=n+2}^{2n} 2(i-n+k-1) \nonumber \\
& =2k+2nk-2(n-1)k+c_3 = 3k +c_3.
\end{align}
Again, for $k$ sufficiently large we have $\frac{1}{2}y=-\frac{1}{2}\sum_{i=1}^{2n}x_i < B(-k-n-1) < \frac{1}{2} x_{2n}$. This implies also in this last case that $\abs{\sigma_{2n+1}(\alpha)}<\abs{\sigma_{2n}(\alpha)}$, which proves the lemma.
\end{proof}

\begin{lemma}\label{lem:Dn}
Let $K$ be a number field of degree $d$ with Galois closure $K_{\mathcal{G}}$ and $\Gal(K_{\mathcal{G}}/\Q)\cong D_d$. Let $\sigma \in \Gal(K_{\mathcal{G}}/\Q)$ be of order $d$, then 
\begin{enumerate}[label=(\roman*)]
\item $K$ is the fixed field of $\langle \tau \rangle$ for some reflection $\tau \in \Gal(K_{\mathcal{G}}/\Q)\cong D_d$,
\item the elements from $\langle \sigma \rangle$ restricted to $K$ form a full set of pairwise distinct embeddings of $K$ into $\mathbb{C}$, and
\item for each $k\in\Z$ and each $\alpha\in K$ we have $\tau\sigma^k(\alpha)=\sigma^{-k}(\alpha)$, for $\tau$ from statement (i).
\end{enumerate}
\end{lemma}
\begin{proof}
Since $[K_\mathcal{G}:K]=2$, $K$ is the fixed field of a subgroup of order $2$. Moreover, it cannot be fixed by the normal subgroup $\langle \sigma^{\nicefrac{d}{2}}\rangle$, since then $K/\Q$ would be Galois which is not the case. All other subgroups of order $2$ are generated by a reflection, proving (i). 

We have just seen that no element from $\sigma^k$, $k\in\{1,\ldots,d-1\}$ fixes $K$. Hence, these elements are precisely the non trivial embeddings of $K$ into $\C$, which settles (ii). By the group structure of $D_d$, we have $\tau\sigma^k=\sigma^{-k}\tau$. Since $\tau$ fixes all elements in $K$, statement (iii) follows.
\end{proof}

\begin{lemma}\label{lem:Fn}
Let $K$ be a number field of degree 5 with Galois closure $K_{\mathcal{G}}$ and $\Gal(K_{\mathcal{G}}/\Q)\cong F_5$. Let $\sigma \in \Gal(K_{\mathcal{G}}/\Q)$ be of order 5, then 
\begin{enumerate}[label=(\roman*)]
\item $K$ is the fixed field of $\langle \tau \rangle$ for some element $\tau$ in  $\Gal(K_{\mathcal{G}}/\Q)$ of order 4, such that $\tau\sigma=\sigma^2\tau$.
\item The elements from $\langle \sigma \rangle$ restricted to $K$ form a full set of pairwise distinct embeddings of $K$ into $\mathbb{C}$.
\item For $k,j\in\Z$, and each $\alpha\in K$ we have $\tau^k\sigma^j(\alpha)=\sigma^{2kj}(\alpha)$, for $\tau$ from (i).
\end{enumerate}
\end{lemma}
\begin{proof}
Since $[K_\mathcal{G}:K]=4$, $K$ is the fixed field of a subgroup of order $4$. Since $C_2^2$ is not a subgroup of $F_5$, $K$ must be the fixed field of $\langle \tau \rangle=\langle \tau^{-1} \rangle$ for some $\tau \in \Gal(K_{\mathcal{G}}/K)$ of order $4$. Now, either $\tau \sigma \tau^{-1} = \sigma^2 $ or $\tau^{-1} \sigma \tau = \sigma^2$, proving (i).

Therefore, no element from $\sigma^i$, $i\in\{1,\ldots,4\}$ fixes $K$. Hence, these are distinct embeddings of $K$ into $\C$, and so the elements of $\langle \sigma \rangle$ are precisely all the embeddings of $K$. This gives us (ii). Now, since $\tau\sigma=\sigma^2\tau$, we have 
\[
 \tau^k\sigma^j(\alpha)=\tau^{k-1}\sigma^{2j}\tau(\alpha)  = \tau^{k-2}\sigma^{4j}\tau^2(\alpha)
 = \tau^{k-k}\sigma^{2kj}\tau^k(\alpha) =\sigma^{2kj}(\alpha).
\]
This completes the proof.
\end{proof}

\section{Classification of quartic number fields}
In this section $K$ will always denote a quartic number field. If $\alpha \in K$ is not a generator of $K$, then the degree of $\alpha$ is at most $2$ and hence $\alpha$ is preperiodic. Therefore, we will only consider the orbits of generators of $K$. We break down our analysis based on the signature of the quartic field $K$, starting with totally imaginary extensions.

\subsection{Totally imaginary fields}
Let $K$ be a totally imaginary quartic number field, and let $\alpha$ be a generator of $K$. Moreover, let $\ell\in\mathbb{N}$ be the leading coefficient of the minimal polynomial of $\alpha$ (note, that we do not assume that $\alpha$ is an algebraic integer). Denote the Galois conjugates of $\alpha$ by $\alpha_1$, $\alpha_2$, $\alpha_3$, $\alpha_4$, such that $\vert \alpha_1 \vert = \vert \alpha_2 \vert > \vert \alpha_3\vert =\vert \alpha_4\vert$. Then the Mahler measure of $\alpha$ is one of the elements $\pm\ell$, $\pm\ell\alpha_1\alpha_2$, $\pm\ell\alpha_1\alpha_2\alpha_3\alpha_4$.

Since $\ell$ and  $\ell\alpha_1\alpha_2\alpha_3\alpha_4=:n$ are in $\Z$, our $\alpha$ is preperiodic, whenever $M(\alpha)\neq \pm\ell\alpha_1\alpha_2$. Hence we will assume
\[
\vert \alpha_1 \vert = \vert \alpha_2 \vert > 1 > \vert \alpha_3\vert =\vert \alpha_4\vert.
\]
The Galois conjugates of $\ell\alpha_1\alpha_2$ all lie in the set
\[
\{\ell\alpha_1\alpha_2,\ell\alpha_1\alpha_3,\ell\alpha_1\alpha_4,\ell\alpha_2\alpha_3,\ell\alpha_2\alpha_4,\ell\alpha_3\alpha_4\}.
\]
Before we proceed, we remark that $\ell\alpha_1\alpha_2$ is a Perron integer.

\begin{lemma}\label{lem:totimaginary}
If $\alpha_1\alpha_3=\alpha_2\alpha_4$ (resp. $\alpha_1\alpha_4=\alpha_2\alpha_3$), then $\alpha_1\alpha_3$ (resp. $\alpha_1\alpha_4$) is not a Galois conjugate of $\alpha_1\alpha_2$.
\end{lemma}
\begin{proof}
Assume $\alpha_1\alpha_3=\alpha_2\alpha_4$, then $(\alpha_1\alpha_3)^2=\alpha_1\alpha_2\alpha_3\alpha_4 \in \Q$. Therefore, $\alpha_1\alpha_3$ cannot be a Galois conjugate of the Perron number $\alpha_1\alpha_2$. The same argument applies if $\alpha_1\alpha_4=\alpha_2\alpha_3$.
\end{proof}

If all Galois conjugates of $\ell\alpha_1\alpha_2$ have absolute value $\geq 1$, then $\ell\alpha_1\alpha_2$ -- and hence $\alpha$ -- is surely preperiodic. So we assume, that $\ell\alpha_3\alpha_4 <1$.

Since in any case we have $\alpha_1\alpha_3=\overline{\alpha_2\alpha_4}$, it follows by Lemma \ref{lem:totimaginary} that
\begin{align*}
 & M^{2}(\alpha)=M(\ell\alpha_1\alpha_2) \\ \in &\{\pm\ell\alpha_1\alpha_2, \pm\ell^3 \alpha_1\alpha_2\alpha_1\alpha_3\alpha_2\alpha_4, \pm\ell^5\alpha_1\alpha_2\alpha_1\alpha_3\alpha_1\alpha_4\alpha_2\alpha_4\alpha_2\alpha_3\}\\
 = &\{\pm\ell\alpha_1\alpha_2, \pm\ell^2 n \alpha_1\alpha_2, \pm\ell^3 n^2 \alpha_1\alpha_2\}
\end{align*}

If $\vert \ell n\vert =1$, which is precisely the case if $\alpha$ is an algebraic unit, then $M^{2}(\alpha)=M(\alpha)$ and $\alpha$ is preperiodic. Assuming that $\vert \ell n\vert >1$, and $M^{2}(\alpha)\neq M(\alpha)$, then the possible Galois conjugates of $M^{2}(\alpha)$ are
\[
\{\pm \ell^i n^j\alpha_1\alpha_2,\pm\ell^i n^j\alpha_1\alpha_3,\pm \ell^i n^j\alpha_1\alpha_4, \pm\ell^i n^j\alpha_2\alpha_3, \pm \ell^i n^j\alpha_3\alpha_4\}, 
\]
with $(i,j)\in \{(2,1),(3,2)\}$. As before, either $\vert \ell^i n^j\alpha_3\alpha_4\vert >1$, and we are done, or $M^{3}(\alpha)$ is an element from
\[
\{\pm \ell^{3i-1}n^{3j+1}\alpha_1\alpha_2,\pm \ell^{5i-2}n^{5j+2}\alpha_1\alpha_2 \}.
\]
Since each iteration increases the absolute value of the rational integer in front of  the $\alpha_1\alpha_2$, after finitely many iterations, we see that all Galois conjugates of some $M^{k}(\alpha)$ lie outside the unit circle. Hence, $\alpha$ is preperiodic in all cases. We conclude:

\begin{prop}\label{prop:totimaginary}
If $K/\Q$ is a totally imaginary quartic number field, then all elements in $K$ are preperiodic under iteration of the Mahler measure.
\end{prop}

\begin{rmk}
Note that under our assumptions it is possible that the degree of $M(\alpha)$ is larger than four. For instance, the Mahler measure of any root of $x^4 -x +1$ is of degree six. Since no root of this polynomial is real, we know that it is a preperiodic point under iteration of $M$.
\end{rmk}

\subsection{Fields of signature $(2,1)$}
Let $K$ have two real embeddings and two complex conjugate embeddings.  Denote by $K_{\mathcal{G}}$ the Galois closure of $K/\Q$. Since $K$ is neither totally real, nor totally imaginary, it cannot be a Galois extension of $\Q$. Hence, the only possibilities for $\Gal(K_{\mathcal{G}}/\Q)$ are $S_4$, $A_4$, or $D_4$. 

Let us first assume that $\Gal(K^{\mathcal{G}}/\Q)$ is isomorphic to $S_4$ or $A_4$. Let $\alpha \in K$ be an algebraic unit such that the Galois conjugates $\alpha_1,\ldots,\alpha_4$ of $K$ are given in the order
\[
\abs{\alpha_1} > \abs{\alpha_2} >1 >\abs{\alpha_3} = \abs{\alpha_4}.
\]
Note that $\alpha_4=\overline{\alpha_3}$. We see that $M(\alpha)= \pm \alpha_1\alpha_2$. Since we assume that $\Gal(K_{\mathcal{G}}/\Q)$ contains $A_4$, the Galois orbit of $\alpha_1\alpha_2$ is as large as possible; i.e. it is equal to the set
\[
\{ \alpha_1\alpha_2, \alpha_1\alpha_3 , \alpha_1\alpha_4, \alpha_2\alpha_3, \alpha_2\alpha_4, \alpha_3\alpha_4\}.
\]
Since $\abs{\alpha_1\alpha_3}\cdot \abs{\alpha_2\alpha_4} = \abs{N(\alpha)}=1$ and $\abs{\alpha_1\alpha_3}> \abs{\alpha_2 \alpha_4}$, it follows that $\abs{\alpha_1\alpha_3} = \abs{\alpha_1\alpha_4}>1$ and $\abs{\alpha_2\alpha_3}=\abs{\alpha_2\alpha_4}<1$. Therefore
\[
M^{2}(\alpha)=M(\alpha_1\alpha_2)= \pm \alpha_1\alpha_2\alpha_1\alpha_3\alpha_1\alpha_4=\pm \alpha_1^2 N(\alpha)=\alpha_1^2.
\]
Since $\alpha$ is torsion-free, it follows that $\alpha$ is a wandering point by Lemma \ref{lem:torsionfree}.

We continue with the case where $\Gal(K_{\mathcal{G}}/\Q)$ is isomorphic to $D_4$. Since every subgroup of $D_4$ of order two is contained in a subgroup of order four, Galois theory tells us that $K$ contains a quadratic subfield. Hence, $K$ is a quadratic extension of a quadratic extension of $\Q$, and therefore generated by an element of the form $\sqrt{a+b\sqrt{d}}$ with $a,b\in\Q$ and $d\in \Z$ squarefree.

We are interested in the Mahler measure of elements in $K$. Since the Mahler measure is invariant under taking Galois conjugates, and $K$ is of signature $(1,2)$, we may assume that $K$ is real. This means, we can assume that $d\in\mathbb{N}$ is squarefree and $a+b\sqrt{d} >0$. Applying once more our assumption on the signature of $K$, it follows that we can assume $K=\Q(\sqrt{a+b\sqrt{d}})$, with $d\in\mathbb{N}$ squarefree, $a,b\in\Q$ such that $a+b\sqrt{d}>0$ and $a-b\sqrt{d}<0$.

Let $\alpha\in K$ be of degree four. Then there are $r,s,t,u\in\Q$ such that the Galois conjugates of $\alpha$ are
\begin{align*}
\alpha_1 & = r+s\sqrt{a+b\sqrt{d}}+t(a+b\sqrt{d})+u\left(\sqrt{a+b\sqrt{d}}\right)^3 \\
\alpha_2 & = r-s\sqrt{a+b\sqrt{d}}+t(a+b\sqrt{d})-u\left(\sqrt{a+b\sqrt{d}}\right)^3 \\
\alpha_3 & = r+s\sqrt{a-b\sqrt{d}}+t(a-b\sqrt{d})+u\left(\sqrt{a-b\sqrt{d}}\right)^3 \\
\alpha_4 & = r-s\sqrt{a-b\sqrt{d}}+t(a-b\sqrt{d})-u\left(\sqrt{a-b\sqrt{d}}\right)^3.
\end{align*}
Let $\ell\in\mathbb{N}$ denote the leading coefficient of the minimal polynomial of $\alpha$. Then
\[
M(\alpha)\in\{\pm \ell \alpha_1,\pm \ell\alpha_2, \pm\ell\alpha_1\alpha_2,\pm \ell\alpha_1\alpha_3\alpha_4, \pm\ell\alpha_2\alpha_3\alpha_4,\pm\ell\alpha_3\alpha_4, \pm\ell\alpha_1\alpha_2\alpha_3\alpha_4\}.
\]
Note, that by assumption $\alpha_3=\overline{\alpha_4}$. Moreover, $\alpha_3\alpha_4,\alpha_1\alpha_2\in\Q(\sqrt{d})$, and $K=\Q(\alpha_1)=\Q(\alpha_2)$. Hence, in any case we have $M(\alpha)\in K$. Therefore, after replacing $\alpha$ by $M(\alpha)$, we may assume that $\alpha=\alpha_1$ is an algebraic integer in $K$ of absolute value $>1$. Again we assume that $\alpha$ is of degree $4$, since otherwise it is surely preperiodic. If $\alpha_2,\alpha_3,\alpha_4$ lie all outside (or all inside) the closed unit circle, $\alpha$ is surely preperiodic. So we assume that one of $\alpha_2,\alpha_3,\alpha_4$ lies inside, and another lies outside the unit circle. Then
\[
M(\alpha)\in\{\pm \alpha_1\alpha_2, \pm \alpha_1\alpha_3\alpha_4\}.
\]
If $M(\alpha)=\pm \alpha_1\alpha_2$, then $M(\alpha) \in \Q(\sqrt{d})$, and $\alpha$ is preperiodic. So we proceed with $M(\alpha)=\pm \alpha_1\alpha_3\alpha_4=\pm\frac{N(\alpha)}{\alpha_2}\in K$. As before we see, that if the number of Galois conjugates of $\frac{N(\alpha)}{\alpha_2}$ outside the unit circle is equal to $1$, $2$ or $4$, then $\alpha$ is preperiodic. But if this number is equal to $3$, then $\alpha$ satisfies the assumption of \cite[Proposition 1]{FPZ2020}, and hence it is preperiodic in this case as well.

We summarize the above in a proposition:
\begin{prop}\label{prop:signature(2,1)}
Let $K$ be a quartic number field of signature $(2,1)$. Then all $\alpha \in K$ are preperiodic under iteration of the Mahler measure, if and only if the Galois group of the Galois closure of $K$ over $\Q$ is isomorphic to the dihedral group $D_4$. In all other cases $K$ contains wandering algebraic units.
\end{prop}

\subsection{Totally real fields}

\begin{lemma}\label{lem:totreal4}
Let $K$ be a totally real number field of degree $4$. Denote by $\sigma_1, \sigma_2,\sigma_3$ the non-trivial embeddings of $K$ into $\mathbb{R}$. Then there exists an algebraic unit $\alpha \in K$, such that
\begin{enumerate}[label=(\roman*)]
\item $\vert \alpha \vert > \vert \sigma_1 (\alpha) \vert >1$,
\item $\vert \sigma_2 ( \alpha) \vert, \vert \sigma_3(\alpha)\vert <1$, and
\item $\vert \alpha \sigma_i(\alpha)\vert \neq 1$ for all $i\in\{1,2,3\}$.
\end{enumerate}
\end{lemma}
\begin{proof}
Let $\mathcal{O}_K^*$ be the unit group of $K$. For $\alpha\in\mathcal{O}_K*$ we define $$L(\alpha)=(\log \vert \alpha\vert, \log\vert \sigma_1(\alpha)\vert ,\log\vert \sigma_2(\alpha)\vert,\log\vert \sigma_3(\alpha)\vert).$$ By Dirichlet's unit theorem $
L(\mathcal{O}_K^*)$ is a lattice of rank $3$ in the hyperplane $$\{(x_1,x_2,x_3,-x_1-x_2-x_3)\vert x_1,x_2,x_3\in\mathbb{R}\} \subseteq \mathbb{R}^4.$$

Let $B\in\mathbb{R}$ be larger than any vector spanning a fixed fundamental domain of the lattice $L(\mathcal{O}_K^*)$. Then there exists an element $(x_1,x_2,x_3,x_4) \in L(\mathcal{O}_K^*)$ such that $x_1 >2B$, $0<x_2<B$, and $-B > x_3 >-2B$. It follows that $x_4=-x_1-x_2-x_3 < -2B+2B=0$ and $x_1+x_3\neq 0 \neq x_1+x_4$. This means that any $\alpha\in \mathcal{O}_K^*$ such that $L(\alpha)=(x_1,x_2,x_3,x_4)$ satisfies the statements (i), (ii), (iii) from the lemma. 
\end{proof}

Let $K$ be a totally real quartic number field. Again we have to distinguish between different cases depending on the structure of $\Gal(K_{\mathcal{G}}/\Q)$, where $K_{\mathcal{G}}$ is the Galois closure of $K$ over $\mathbb{Q}$.

First we assume that $\Gal(K_{\mathcal{G}}/\Q)$ is isomorphic to $S_4$ or $A_4$. In this case the existence of wandering units follows precisely as in the case of signature $(2,1)$: By Lemma \ref{lem:totreal4} there is an algebraic unit in $K$ with Galois conjugates $\alpha_1,\alpha_2,\alpha_3,\alpha_4$ ordered as
\[
\abs{\alpha_1}\geq\abs{\alpha_2} >1 > \abs{\alpha_3} \geq \abs{\alpha_4}
\]
and $\abs{\alpha_1\alpha_4} \neq 1$. By replacing $\alpha$ with $\alpha^{-1}$ if necessary, we can assume $\abs{\alpha_1\alpha_4}>1$, which immediately implies $\abs{\alpha_2\alpha_3}<1$. By our assumption on the Galois group $\Gal(K_{\mathcal{G}}/\Q)$, we know that the degree of $M(\alpha)=\pm\alpha_1\alpha_2$ is equal to six. It follows
\[
M^{2}(\alpha)=M(\alpha_1\alpha_2)=\pm \alpha_1\alpha_2\alpha_1\alpha_3\alpha_1\alpha_4=\alpha_1^2.
\]
Since $\alpha$ is torsion-free, it follows that $\alpha$ is wandering unit under iteration of $M$.

We proceed with the assumption that $\Gal(K_{\mathcal{G}}/\Q)$ is isomorphic to $C_4$ or $D_4$. Let $\sigma\in \Gal(K_{\mathcal{G}}/\Q)$ be of order $4$. By Lemmas \ref{lem:Dn} and \ref{lem:totreal4} there is an algebraic unit $\alpha \in K$ of degree $4$ such that
\begin{equation}\label{eq:totallyrealcyclic}
\vert \alpha \vert > \vert \sigma(\alpha) \vert > 1 \quad \text{ and } \quad \vert \sigma^{2}(\alpha) \vert,\ \vert \sigma^{3}(\alpha)\vert<1 \quad \text{ and } \quad \vert \alpha\sigma^{3}(\alpha)\vert \neq 1.
\end{equation}
Now, the Mahler measure of $\alpha$ is $M(\alpha)=\pm \alpha\sigma(\alpha)$. If the Galois group is $D_4$, let $\tau$ be the reflexion fixing $K$. Then, by Lemma \ref{lem:Dn},
\[
\tau(\alpha\sigma(\alpha))=\alpha \sigma^{-1}(\alpha)=\sigma^{-1}(\alpha\sigma(\alpha)).
\]
Hence, in any case, all Galois conjugates of $\alpha\sigma(\alpha)$ are given by
\[
 \alpha\sigma(\alpha), \quad \sigma(\alpha)\sigma^{2}(\alpha), \quad \sigma^{2}(\alpha)\sigma^{3}(\alpha), \quad\text{and}\quad \sigma^{3}(\alpha)\alpha.
\]
Since $\alpha$ is an algebraic unit,
\[
\vert \sigma(\alpha)\sigma^{2}(\alpha)\vert \cdot \vert \sigma^{3}(\alpha)\alpha \vert =1, 
\]
and by \eqref{eq:totallyrealcyclic} we have $\vert \sigma^{3}(\alpha)\alpha \vert \neq 1$. It follows, that precisely one of $ \sigma^{3}(\alpha)\alpha$ and $\sigma(\alpha)\sigma^{2}(\alpha)$ lies outside the unit circle. Hence $M^{2}(\alpha)=M(\alpha\sigma(\alpha))> M(\alpha)$, and the orbit of $\alpha$ under iteration of $M$ contains at least three elements. It follows from \cite[Theorem 2]{FPZ2020}, that $\alpha$ is a wandering unit.

The last remaining case is $K_{\mathcal{G}}=K$, and $\Gal(K/\Q)$ is isomorphic to $C_2\times C_2$. This is precisely the case when $K=\Q(\sqrt{p},\sqrt{q})$ is a biquadratic field, with $p,q\in\mathbb{N}$ coprime. Since $K/\Q$ is Galois, we have $M(\alpha) \in K$ for all $\alpha\in K$. Hence, we may replace $\alpha$ by $M(\alpha)$ and assume that $\alpha>1$ is an algebraic integer in $K$. There are $a,b,c,d \in\Q$ such that the Galois conjugates of $\alpha$ are
\begin{align*}
\alpha_1 = a+b\sqrt{p} + c\sqrt{q} +d \sqrt{pq} \qquad & \alpha_2=a-b\sqrt{p} + c\sqrt{q} -d \sqrt{pq}\\
\alpha_3 = a+b\sqrt{p} - c\sqrt{q} -d \sqrt{pq} \qquad & \alpha_4=a-b\sqrt{p} - c\sqrt{q} +d \sqrt{pq}.
\end{align*}
If all of these elements, or if precisely one of these elements, lie outside the unit circle, then $\alpha$ is obviously preperiodic. If precisely two conjugates lie outside the unit circle, then $M(\alpha)$ is contained in a quadratic field, and hence $\alpha$ is preperiodic in this case as well. Lastly, we assume that precisely three conjugates of $\alpha$ lie outside the unit circle. Without loss of generality, let these conjugates be $\alpha_1,\alpha_2,\alpha_3$. Then $M(\alpha)=\pm \alpha_1\alpha_2\alpha_3 = \pm N(\alpha)/\alpha_4 \in K$. Again, $\alpha$ is preperiodic if precisely $1$, $2$, or $4$ conjugates of $M(\alpha)$ lie outside the unit circle. But if precisely $3$ conjugates of $M(\alpha)$ lie outside the unit circle, then $\alpha$ satisfies the assumptions from \cite[Proposition 1]{FPZ2020}, which implies that $\alpha$ is also in this case a preperiodic point. 

Now we can also summarize the results on wandering points in totally real quartic extensions:
\begin{prop}\label{prop:totreal4}
Let $K$ be a totally real quartic number field. Then all points in $K$ are preperiodic under iteration of $M$ if and only if $K$ is biquadratic. In all other cases there are wandering units in $K$.
\end{prop}

All of the results of this section are summarized in Table \ref{tab1} of the introduction. 

\section{Classification of quintic number fields}
In this section $K$ always denotes a number field of degree $5$. We will show that in these number fields, there always exists an algebraic unit which is wandering under iteration of $M$. A partial result for this claim can be obtained immediately. As before we denote the Galois closure of $K/\Q$ by $K_{\mathcal{G}}$. 

\begin{prop}\label{prop:deg5A5S5}
If the Galois group $K_{\mathcal{G}}/\Q$ is isomorphic to $A_5$ or $S_5$, then $K$ contains a wandering unit. In particular, there is a wandering unit in $K$ if the signature of $K$ is $(3,1)$.
\end{prop}
\begin{proof}
By \cite[Theorem 3]{FPZ2020} any algebraic unit $\alpha \in K$ such that none of $\pm \alpha^{\pm 1}$ is a Pisot-Vijayaraghavan number is wandering under iteration of $M$. By Dirichlet's unit theorem, we can find such an algebraic unit in $K$. Lastly, we note that if the signature of $K$ is $(3,1)$, then $\Gal(K_{\mathcal{G}}/\Q)$ is isomorphic to the full group $S_5$ (see \cite[Lemma 15.9]{Stewart}), which completes the proof of the proposition.
\end{proof}

Since $\Gal(K_{\mathcal{G}}/\Q)$ must be isomorphic to a transitive subgroup of $S_5$, it remains to construct wandering units if $\Gal(K_{\mathcal{G}}/\Q)$ is isomorphic to $C_5$, $D_5$, or $F_5$.

\begin{prop}\label{prop:cyclic5}
Let $K/\mathbb{Q}$ be a Galois extension of degree five. Then there is an algebraic unit in $K$, which is a wandering point under iteration of the Mahler measure.
\end{prop}
\begin{proof}
If $K/\Q$ is a Galois extension of degree five, then $\Gal(K/\Q)\cong C_5$. The exsistence of wandering units in this case is proved later in Section \ref{sec:cyclic}. 
\end{proof}

\begin{lemma}\label{lem:D5signature(1,2)}
Let $K$ be a real number field of degree five with signature $(1,2)$. Let $K_{\mathcal{G}}$ denote the Galois closure of $K$ over $\Q$, and assume $\Gal(K_{\mathcal{G}}/\Q)\cong D_5$ or $\Gal(K_{\mathcal{G}}/\Q)\cong F_5$. For any $\sigma \in \Gal(K_{\mathcal{G}}/\Q))$ of order five, there is an algebraic unit $\alpha \in K$ satisfying
\begin{enumerate}[label=(\roman*),start=1]
\item $\abs{\alpha} > \abs{\sigma(\alpha)} = \abs{\sigma^4(\alpha)}>1 > \abs{\sigma^2(\alpha)} = \abs{\sigma^3(\alpha)}$, and
\item $\abs{\alpha \sigma^2(\alpha)} <1$.
\end{enumerate}
\end{lemma}
\begin{proof}
Let $\sigma \in \Gal(K_{\mathcal{G}}/\Q))$ be of order five. Since $K$ is real, it is fixed by complex conjugation. If $\Gal(K_{\mathcal{G}}/\Q)\cong D_5$, then complex conjugation must be given by $\tau$ from Lemma \ref{lem:Dn}. Moreover, $\tau\sigma(\alpha)=\sigma^4\tau(\alpha)=\sigma^4(\alpha)$. In the case that $\Gal(K_{\mathcal{G}}/\Q)\cong F_5$, complex conjugation must be an element of order 2 fixing $K$. Hence, complex conjugation is given by $\tau^2$ for the automorphism $\tau$ from Lemma \ref{lem:Fn}. Moreover, $\tau^2\sigma(\alpha)=\sigma^4\tau^2(\alpha)=\sigma^4(\alpha)$. Now, in every case, it follows by Lemma \ref{lem:deg5} that there is an algebraic unit $\alpha \in K$ such that $\abs{\alpha}> \abs{\sigma(\alpha)}=\abs{\sigma^4(\alpha)}>1 > \abs{\sigma^2(\alpha)}= \abs{\sigma^3(\alpha)}$, proving the lemma.
\end{proof}

\begin{prop}\label{prop:Dih5}
Suppose $K/\mathbb{Q}$ is an extension of degree five. Let $K_{\mathcal{G}}$ be the Galois closure of $K$ over $\Q$. If $\Gal(K_{\mathcal{G}}/\Q)\cong D_5$, then there is an algebraic unit in $K$ which is a wandering point under iteration of the Mahler measure.
\end{prop}

\begin{proof}
Note that under our assumption on the Galois group the only possible signatures of $K$ are $(5,0)$ or $(1,2)$. 
Let $\sigma \in \Gal(K_{\mathcal{G}}/\Q)$ be of order $5$, and let $\tau$ be as in Lemma \ref{lem:Dn}. Then by Lemma \ref{lem:Dn}, $\id$, $\sigma\vert_K$, $\sigma^2\vert_K$, $\sigma^3\vert_K$, and $\sigma^4\vert_K$ are precisely all the distinct embeddings of $K$ into $\mathbb{C}$. Since the Mahler measure is invariant under taking Galois conjugates and there is at least one real embedding of $K$, we may assume that $K$ is a real number field. Then, by Lemmas \ref{lem:deg5} and \ref{lem:D5signature(1,2)}, there exists an algebraic unit $\alpha \in K$ satisfying the following conditions:
\begin{enumerate}[label=\alph*.,start=1]
\item $\vert \alpha \vert > \vert \sigma (\alpha) \vert\geq\vert \sigma^4 (\alpha) \vert >1>\vert \sigma^2 (\alpha) \vert \geq \vert \sigma^3 (\alpha) \vert$,
\item $\vert \sigma^2(\alpha)\sigma^4(\alpha)\vert<\vert \alpha\sigma^3 ( \alpha) \vert<1$.
\end{enumerate}

Let $j\geq i \geq 1$ be integers. We will first show that the element $\alpha^j \sigma(\alpha)^i\sigma^4(\alpha)^i$ is not a fixed point for $M$.

We note that the degree of $\alpha^j \sigma(\alpha)^i\sigma^4(\alpha)^i$ is five and all conjugates of $\alpha^j\sigma(\alpha)^i\sigma^4(\alpha)^i$ are constructed by $\sigma^k$ for $k \in \{0,1,2,3,4\}$:
\begin{align*}
 \tau\sigma^k(\alpha^j\sigma(\alpha)^i\sigma^4(\alpha)^i)&=\sigma^{-k}\tau(\alpha)^j\sigma^{-k-1}\tau(\alpha)^i\sigma^{-k-4}\tau(\alpha)^i 
 \\ &= \sigma^{-k}(\alpha)^j\sigma^{-k-1}(\alpha)^i\sigma^{-k-4}(\alpha)^i \\
 &=  \sigma^{-k}(\alpha^j\sigma(\alpha)^i\sigma^4(\alpha)^i)
\end{align*}
Hence all conjugates of $\alpha^i\sigma(\alpha)^j\sigma^2(\alpha)^i$ are given by the action of $\id$, $\sigma$, $\sigma^2$, $\sigma^3$, and $\sigma^4$. 
Now, if $\alpha$ is an algebriac unit satisfying conditions (a) and (b) above, then 
\begin{itemize}
\item $\abs{\alpha^j \sigma(\alpha)^i\sigma^4(\alpha)^i} = \left\vert \frac{\alpha^{j-i}}{\sigma^2(\alpha)^i\sigma^3(\alpha)^i} \right\vert \geq  \left\vert \frac{1}{\sigma^2(\alpha)^i\sigma^3(\alpha)^i} \right\vert >1$
\item $\abs{\sigma\left( \alpha^j \sigma(\alpha)^i\sigma^4(\alpha)^i \right)} = \left\vert \frac{\sigma(\alpha)^{j-i}}{\sigma^3(\alpha)^i\sigma^4(\alpha)^i} \right\vert \geq \left\vert \frac{1}{\sigma^3(\alpha)^i\sigma^4(\alpha)^i} \right\vert >1$
\item $\abs{\sigma^2\left( \alpha^j \sigma(\alpha)^i\sigma^4(\alpha)^i \right)} = \left\vert \frac{\sigma^2(\alpha)^{j-i}}{\alpha^i\sigma^4(\alpha)^i} \right\vert \leq \left\vert \frac{1}{\alpha^i\sigma^4(\alpha)^i} \right\vert <1$
\item $\abs{\sigma^3\left( \alpha^j \sigma(\alpha)^i\sigma^4(\alpha)^i \right)} = \left\vert \frac{\sigma^3(\alpha)^{j-i}}{\sigma(\alpha)^i\alpha^i} \right\vert \leq \left\vert \frac{1}{\sigma(\alpha)^i\alpha^i} \right\vert <1$
\item $\abs{\sigma^4\left( \alpha^j \sigma(\alpha)^i\sigma^4(\alpha)^i \right)} = \left\vert \frac{\sigma^4(\alpha)^{j-i}}{\sigma^2(\alpha)^i\sigma(\alpha)^i} \right\vert \geq \left\vert \frac{1}{\sigma^2(\alpha)^i\sigma(\alpha)^i} \right\vert >1$
\end{itemize}

This shows that $\alpha^j\sigma(\alpha)^i\sigma^4(\alpha)^i$ is not a fixed point for $M$, since there are non trivial Galois conjugates outside the unit circle. Moreover, we have that
\[
M(\alpha^j \sigma(\alpha)^i\sigma^4(\alpha)^i)= \abs{\alpha^{j+2i}\sigma(\alpha)^{j+i}\sigma^2(\alpha)^i \sigma^3(\alpha)^i\sigma^4(\alpha)^{j+i}}=\abs{\alpha^{j+i}\sigma(\alpha)^j \sigma^4(\alpha)^j }
\]
is again of the prescribed form, and hence it is not a fixed point. Inductively, it follows that $\alpha^j\sigma(\alpha)^i\sigma^4(\alpha)^i$ is a wandering point under iteration of $M$. Since $M(\alpha)=\pm \alpha \sigma(\alpha)\sigma^4(\alpha)$, we see that $\alpha$ is a wandering point as well.
\end{proof}

\begin{prop}\label{prop:F5}
Suppose $K/\Q$ is of degree five and let $K_{\mathcal{G}}$ be the Galois closure of $K$ over $\Q$. If $\Gal(K_{\mathcal{G}}/\Q)$ is isomorphic to the Frobenius group $F_5$, then there exists an algebraic unit in $K$ which is wandering under iteration of the Mahler measure $M$.
\end{prop}
\begin{proof}
Again we note, that the assumption on the Galois group implies that the signature of $K$ is either $(5,0)$ or $(1,2)$. Let $\sigma \in \Gal(K_{\mathcal{G}}/\Q)$ be of order $5$, and let $\tau$ be as in Lemma \ref{lem:Fn}. Then by Lemma \ref{lem:Fn}, $\id$, $\sigma\vert_K$, $\sigma^2\vert_K$, $\sigma^3\vert_K$, and $\sigma^4\vert_K$ are precisely all of the distinct embeddings of $K$ into $\mathbb{C}$. Again, we assume without loss of generality that $K$ is a real number field. Then by Lemma \ref{lem:deg5} and Lemma \ref{lem:D5signature(1,2)} there exists an algebraic unit $\alpha \in K$ satisfying the following conditions:
\begin{enumerate}[label=\alph*.,start=1]
\item $\abs{\alpha} > \abs{\sigma(\alpha)} \geq \abs{\sigma^4(\alpha)} >1 >\abs{\sigma^{2}(\alpha)} \geq \abs{\sigma^3(\alpha)}$,
\item $\abs{\alpha\sigma^{2}(\alpha)}<1$.
\end{enumerate}
For such an $\alpha\in K$ we have  $M(\alpha)=\pm\alpha\sigma(\alpha)\sigma^4(\alpha)$. Since $\tau^2$ fixes $\alpha$ and swaps $\sigma(\alpha)$ and $\sigma^4(\alpha)$, we have that $M(\alpha)$ is fixed by $\tau$. In particular, the degree of $M(\alpha)$ is at most $10$. Since the elements $\sigma^k,\sigma^k\tau$ for $k\in\{0,\ldots,4\}$ give $10$ different Galois conjugates of $M(\alpha)$, the degree is equal to $10$. In order to write down the Galois cojugates explicitly, we first calculate (using Lemma \ref{lem:Fn}) 
\[
\tau(\alpha\sigma(\alpha)\sigma^4(\alpha))=  \alpha \sigma^2(\alpha)\sigma^3(\alpha).
\]
Now, we easily calculate the absolute values of the Galois conjugates of $M(\alpha)$:
$$\vert \alpha\sigma(\alpha)\sigma^4(\alpha)\vert=\left\vert \frac{1}{\sigma^2(\alpha)\sigma^3(\alpha)} \right\vert>1, \quad \vert \alpha\sigma^2(\alpha)\sigma^3(\alpha)\vert=\left\vert \frac{1}{\sigma(\alpha)\sigma^4(\alpha)} \right\vert<1,$$
$$\vert \sigma(\alpha)\sigma^2(\alpha)\alpha\vert=\left\vert \frac{1}{\sigma^3(\alpha)\sigma^4(\alpha)} \right\vert>1, \quad  \vert \sigma(\alpha)\sigma^3(\alpha)\sigma^4(\alpha)\vert=\left\vert \frac{1}{\alpha\sigma^2(\alpha)} \right\vert>1, $$
$$\vert \sigma^2(\alpha)\sigma^3(\alpha)\sigma(\alpha)\vert=\left\vert \frac{1}{\alpha\sigma^4(\alpha)} \right\vert<1, \quad  \vert \sigma^2(\alpha)\sigma^4(\alpha)\alpha\vert=\left\vert \frac{1}{\sigma(\alpha)\sigma^3(\alpha)} \right\vert>1,$$ $$\vert \sigma^3(\alpha)\sigma^4(\alpha)\sigma^2(\alpha)\vert=\left\vert \frac{1}{\alpha\sigma(\alpha)} \right\vert<1, \quad  
\vert \sigma^3(\alpha)\alpha\sigma(\alpha)\vert=\left\vert \frac{1}{\sigma^2(\alpha)\sigma^4(\alpha)} \right|>1,$$ $$\vert \sigma^4(\alpha)\alpha\sigma^3(\alpha)\vert=\left\vert \frac{1}{\sigma(\alpha)\sigma^2(\alpha)} \right|>1, \quad  
\vert \sigma^4(\alpha)\sigma(\alpha)\sigma^2(\alpha)\vert=\left\vert \frac{1}{\alpha\sigma^3(\alpha)} \right\vert>1$$
Multiplying the seven terms which are greater than $1$, and using that the norm of $\alpha$ is $\pm1$, gives 
\[
M^2(\alpha) =\pm \alpha^5 \sigma(\alpha)^5\sigma^2(\alpha)^3\sigma^3(\alpha)^3 \sigma^4(\alpha)^5= \pm \alpha^2 \sigma(\alpha)^2\sigma^4(\alpha)^2=M(\alpha)^2.
\]
By Lemma \ref{lem:torsionfree}, this means that $\alpha$ is a wandering unit, proving the proposition.
\end{proof}

To summarize, we get the following theorem:
\begin{thm}
Let $K/\Q$ be an extension of degree five. Then there exists an algebraic unit in $K$ which is wandering under iteration of the Mahler measure $M$.
\end{thm}
\begin{proof}
 Since $\Gal(K_{\mathcal{G}}/\Q)$ must be a transitive subgroup of $S_5$ with order divisible by $5$, it must be isomorphic to $C_5$, $D_5$, $F_5$, $A_5$ or $S_5$ (we note that $\Gal(K_{\mathcal{G}}/\Q)$ cannot be isomorphic to $C_5$ if the signature of $K/\Q$ is (1,2)). In all of these cases, there exists a wandering unit by Propositions \ref{prop:deg5A5S5}--\ref{prop:F5}.
\end{proof}

\section{Some extensions of degree six}

In this section, we will show that there are number fields $K$ of degree $6$ over the rationals such that all elements in $K$ are preperiodic under iteration of $M$, as well as extensions which contain wandering points for the Mahler measure.

Assume that $K/\Q$ is a totally imaginary Galois extension of degree $6$. Then, for any $\alpha\in K$, the value $M(\alpha)$ is a real algebraic integer in $K$. Therefore, $M(\alpha)$ must live in a proper subfield of $K$, which implies that the degree of $M(\alpha)$ is $1$, $2$ or $3$. In all cases it follows that $\alpha$ is preperiodic under iteration of $M$.

However, there are also non-Galois extensions of degree $6$, without wandering points of $M$.

\begin{prop}\label{lem:normalizer6}
Let $K/\Q$ be a CM-field of degree six. Then every element of $K$ is preperiodic under iteration of $M$.
\end{prop}
\begin{proof}
Denote by $\tau$ the element in the Galois group of the Galois closure of $K$ that is given by complex conjugation. Since $K$ is a CM-field, we have $\tau\sigma(\alpha) \in \sigma(K)$ for all $\alpha \in K$ and all $\sigma$ in the associated Galois group. 

Let $\alpha\in K$ be arbitrary. If $\alpha$ lies in a proper subfield of $K$, then the degree of $\alpha$ is at most three and hence $\alpha$ is preperiodic under iteration of $M$. So we assume that $\alpha$ is of degree six. Since $K$ is totally imaginary, we have that either $0$, $2$ ,$4$, or $6$ conjugates of $\alpha$ lie outside the unit circle. If all or no conjugates lie outside the unit circle, then $M(\alpha)$ is a rational integer and hence $\alpha$ is preperiodic under iteration of $M$. By replacing $\alpha$ by $\alpha^{-1}$ if necessary, we may assume that precisely $2$ conjugates of $\alpha$ lie outside the unit circle. Since $K$ is totally imaginary, it follows that $M(\alpha)=\ell\sigma(\alpha)\tau\sigma(\alpha)$ for some embedding $\sigma$, and some $\ell \in \mathbb{Z}$. Hence $M(\alpha)$ is a real element in the field $\sigma(K)$, which means that it lives in a proper subfield of a number field of degree $6$. As before we conclude that $\alpha$ must be preperiodic under iteration of $M$.
\end{proof}

\begin{rmk}
This proposition shows that in general it is not possible to classify fields in which all elements are preperiodic under iteration of $M$ solely in terms of the Galois group and the signature. Let $K$ be the field generated by some root of $x^6 + 6x^4 + 8x^2 + 1$, and denote with $K_{\mathcal{G}}$ its Galois closure over $\Q$. Then $K$ is totally imaginary and $\Gal(K_{\mathcal{G}}/\Q)\cong C_2\times S_4$. Moreover, $K$ is a CM-field, and hence all elements in $K$ are preperiodic under iteration of $M$ by the preceding proposition.

On the other hand, let $L$ be the field generated by some root $\alpha$ of $x^6 + 2x^5 + 3x^4 - 4x^3 + 3x^2 + 2x + 1$, and denote by $L_{\mathcal{G}}$ the Galois closure of $L$ over $\Q$. Then $L$ is again totally imaginary (and in particular, it has the same signature as $K$), and $\Gal(L_{\mathcal{G}}/\Q)\cong \Gal(K_{\mathcal{G}}/\Q)\cong C_2\times S_4$. However, in this case $L$ is \emph{not} a CM-field, and indeed, it is easy to check that $\alpha$ satisfies $M^4(\alpha)=M^2(\alpha)^{20}$, and hence, by Lemma \ref{lem:torsionfree}, $\alpha \in L$ is a wandering point.
\end{rmk}

We have already noted that in a totally imaginary Galois extension of degree $6$ all elements are preperiodic under iteration of $M$. We proceed by studying the other Galois extensions of degree $6$.

\begin{prop}\label{prop:Gal6}
Let $K/\Q$ be Galois of degree $6$. Then there is a wandering unit in $K$ if and only if $K$ is totally real.
\end{prop}
\begin{rmk}
Note that a Galois extension of $\Q$ is either totally real or totally imaginary. Hence, this proposition gives a full classification of Galois extensions $K/\Q$ of degree $6$ where all elements are preperiodic.
\end{rmk}
\begin{proof}
By our preliminary discussion, it remains to prove that in a totally real Galois extension $K/\Q$ of degree $6$, there are wandering units. We have to distinguish between two cases, depending on the Galois group.

We start by considering the case $\Gal(K/\Q)\cong C_6$ and denote by $\sigma$ a generator of $\Gal(K/\Q)$. Let $\alpha$ be an algebraic unit in $K$ such that the Galois conjugates of $\alpha$ satisfy
\begin{equation}\label{eq:cyclic6}
\vert\alpha\vert > \vert \sigma^5(\alpha) \vert > \vert\sigma( \alpha)\vert  > 1 > \vert \sigma^4(\alpha) \vert > \vert \sigma^2\alpha) \vert > \vert \sigma^3(\alpha) \vert.
\end{equation}
That such an algebraic unit exists follows immediately from Dirichlet's unit theorem.

We have $M(\alpha)= \vert \alpha \sigma^5(\alpha) \sigma(\alpha) \vert$. We set $M(\alpha)=\beta$ and note that $\beta$ is again an algebraic unit in $K$. Calculating the Galois conjugates of $\beta$ gives
\begin{align*}
&\vert \beta \vert = \vert \alpha \sigma^5(\alpha) \sigma(\alpha) \vert, \quad &\vert \sigma^5(\beta)\vert = \vert \sigma(\alpha)\sigma^4(\alpha) \alpha \vert,\\ &\vert \sigma(\beta) \vert = \vert \sigma(\alpha) \alpha \sigma^2(\alpha) \vert, \quad &\vert \sigma^4(\beta) \vert = \vert \sigma^4(\alpha) \sigma^3(\alpha) \sigma^5(\alpha)\vert, \\ &\vert \sigma^2(\beta) \vert = \vert \sigma^2(\alpha) \sigma(\alpha) \sigma^3(\alpha)\vert, \quad &\vert \sigma^3(\beta) \vert = \vert \sigma^3(\alpha) \sigma^2(\alpha) \sigma^4(\alpha)\vert.
\end{align*}
By \eqref{eq:cyclic6} we immediately get
$$\vert \beta \vert > \vert \sigma^5(\beta) \vert > \vert \sigma(\beta)\vert > \vert \sigma^4(\beta)\vert > \vert \sigma^2(\beta) \vert > \vert \sigma^3(\beta) \vert. $$
Since $\alpha$ is an algebraic unit, we have $\vert \sigma(\beta) \vert \cdot \vert \sigma^4(\beta) \vert = \vert N(\alpha)\vert =1$, and hence, $\vert \sigma(\beta) \vert > 1 > \vert \sigma^4(\beta)\vert$. We have shown that $M(\alpha)=\beta$ has the same distribution of Galois conjugates as $\alpha$. It follows by induction that $\alpha$ is a wandering point.

The second case is $\Gal(K/\Q)\cong D_3$, but the proof is essentially the same. We denote by $\sigma$ and $\tau$ the generators of $\Gal(K/\Q)$ with the usual properties $\sigma^3 = \tau^2 =\id$ and $\tau \sigma\tau = \sigma^{-1}=\sigma^2$. 

Let again $\alpha$ be an algebraic unit in $K$ satisfying the distribution
\begin{equation}\label{eq:D6}
\vert\alpha\vert > \vert \tau\sigma(\alpha) \vert > \vert \tau\sigma^2(\alpha)\vert  > 1 > \vert \sigma(\alpha) \vert > \vert \sigma^2(\alpha) \vert > \vert \tau(\alpha) \vert,
\end{equation}
which exists thanks to Dirichlet's unit theorem. Then $\beta=M(\alpha)=\lvert \alpha \tau\sigma(\alpha)\tau\sigma^2(\alpha)\rvert$ is an algebraic unit in $K$, satisfying
\begin{align*}
&\vert \beta \vert = \vert \alpha \tau\sigma(\alpha)\tau\sigma^2(\alpha)\vert, \quad &\vert \tau\sigma(\beta) \vert = \vert \tau\sigma(\alpha) \alpha \sigma(\alpha) \vert, \\ &\vert \tau\sigma^2(\beta) \vert = \vert \tau\sigma^2(\alpha) \sigma^2(\alpha) \alpha \vert, \quad &\vert \sigma(\beta) \vert = \vert \sigma(\alpha) \tau(\alpha) \tau\sigma(\alpha)\vert, \\ &\vert \sigma^2(\beta) \vert = \vert \sigma^2(\alpha) \tau\sigma^2(\alpha) \tau(\alpha)\vert, \quad &\vert \tau(\beta) \vert = \vert \tau(\alpha) \sigma(\alpha) \sigma^2(\alpha)\vert.
\end{align*}
By the distribution of $\alpha$ in \eqref{eq:D6} and the fact that $\vert\tau\sigma^2(\beta)\vert \cdot \vert \sigma(\beta)\vert=1$, it follows that
\[
\vert\beta\vert > \vert \tau\sigma(\beta) \vert > \vert \tau\sigma^2(\beta)\vert  > 1 > \vert \sigma(\beta) \vert > \vert \sigma^2(\beta) \vert > \vert \tau(\beta) \vert.
\]
Hence, $\beta=M(\alpha)$ satisfies the same distribution of Galois conjugates as $\alpha$. Inductively it follows that $\alpha$ is a wandering unit under iteration of $M$. This proves the proposition.
\end{proof}

\section{Galois extensions of degree eight}

Assume that $K/\Q$ is a totally imaginary Galois extension of degree $8$. Then, for any $\alpha\in K$ the element $M(\alpha)$, as a real element, lives in a proper subfield of $K$. Hence, if there are no wandering points in the degree four subfield(s) of $K$, then there are no wandering points in $K$. Since we have a full classification of fields of degree $4$ in which no element is wandering, this gives a full classification of totally imaginary Galois extensions of degree $8$ in which no element is wandering.

Let $K/\Q$ be a totally real Galois extension. The Galois group is isomorphic to one of the groups $C_8$, $C_4\times C_2$, $D_4$, $Q_8$, or $C_2\times C_2 \times C_2$. In the first two cases, $K$ has a totally real subfield with Galois group $C_4$, and hence there are wandering units in $K$ by Proposition \ref{prop:totreal4}. In the $D_4$ case there are totally real subfields of degree $4$ with Galois group $D_4$. Again we can conclude by Proposition \ref{prop:totreal4} that $K$ contains wandering algebraic units. In the following we will study the other two remaining cases.

We first consider totally real $C_2^3$-extensions of $\Q$. These are fields of the form $K=\Q(\sqrt{d_1},\sqrt{d_2},\sqrt{d_3})$, with $d_1,d_2,d_3 \geq 2$ squarefree. For each $i\in\{1,2,3\}$ choose a fundamental unit $\beta_i \in \Q(\sqrt{d_i})$ with $\vert \beta_i\vert >1$. Moreover, let $n_1,n_2,n_3$ be positive integers, and set $\alpha_i=\beta_i^{n_i}$. Denote by $\sigma_i$ the nontrivial element in the Galois group of $\Q(\sqrt{d_i})/\Q$. Then $\vert \alpha_i\sigma_i(\alpha_i)\vert =1$ for all $i\in\{1,2,3\}$.  The Galois conjugates of $\alpha=\beta_1^{n_1}\beta_2^{n_2}\beta_3^{n_3}$ are
\begin{center}
\begin{tabular}{cccc}
$\alpha_1\alpha_2\alpha_3$  & $\sigma_1(\alpha_1)\alpha_2\alpha_3$  & $ \alpha_1 \sigma_2(\alpha_2) \sigma_3$  & $ \alpha_1\alpha_2\sigma_3(\alpha_3)$ \\
$\sigma_1(\alpha_1)\sigma_2(\alpha_2)\alpha_3$  & $ \sigma_1(\alpha_1)\alpha_2\sigma_3(\alpha_3)$  & $\alpha_1 \sigma_2(\alpha_2) \sigma_3(\sigma_3)$ & $\sigma_1(\alpha_1)\sigma_2(\alpha_2)\sigma_3(\alpha_3)$
\end{tabular}
\end{center}
We now choose the exponents $n_i$ such that $\vert \beta_1^{n_1}\vert$,  $\vert \beta_2^{n_2}\vert$, and $\vert \beta_3^{n_3}\vert$ are approximately of the same size. More precisely, we choose $n_1,n_2,n_3 \in \mathbb{N}$ such that the product of two of the values $\vert \beta_1^{n_1}\vert, \vert \beta_2^{n_2}\vert ,\vert \beta_3^{n_3}\vert$ is larger than the remaining one. (If we set $b_i = \log\abs{\beta_i}$ then we can assume without loss of generality that $\lceil b_1^{-1} \rceil b_1 < \lceil b_2^{-1} \rceil b_2 < \lceil b_3^{-1} \rceil b_3$, and then one can check that a choice of $n_1= \lceil b_1^{-1}\rceil (\lceil b_2 \rceil + \lceil b_3 \rceil), n_2 = \lceil b_2^{-1} \rceil \lceil b_3 \rceil, n_3 =\lceil b_3^{-1} \rceil \lceil b_2 \rceil$ suffices.) Then we get
\begin{itemize}
\item $\vert \alpha_1\alpha_2\alpha_3 \vert>1$,
\item $\vert \sigma_1(\alpha_1)\alpha_2\alpha_3\vert , \vert \alpha_1 \sigma_2(\alpha_2) \alpha_3 \vert , \vert \alpha_1\alpha_2\sigma_3(\alpha_3) \vert >1$,
\item $\vert \sigma_1(\alpha_1)\sigma_2(\alpha_2)\alpha_3 \vert , \vert \sigma_1(\alpha_1)\alpha_2\sigma_3(\alpha_3) \vert , \vert \alpha_1 \sigma_2(\alpha_2) \sigma_3(\alpha_3)\vert <1$, and
\item $\vert \sigma_1(\alpha_1)\sigma_2(\alpha_2)\sigma_3(\alpha_3) \vert <1$.
\end{itemize}
In particular, $\alpha$ is a Perron number (and hence torsion-free) and the Mahler measure of $\alpha$ is
\[
M(\alpha)=\vert (\alpha_1\alpha_2\alpha_3)(\sigma_1(\alpha_1)\alpha_2\alpha_3)(\alpha_1 \sigma_2(\alpha_2) \alpha_3)(\alpha_1\alpha_2\sigma_3(\alpha_3))\vert = \vert \alpha_1^2 \alpha_2^2 \alpha_3^2 \vert =\alpha^2. 
\]
It follows from Lemma \ref{lem:torsionfree} that $\alpha$ is a wandering algebraic unit under iteration of $M$.

This proves the following proposition:

\begin{prop}\label{prop:3-quad}
Let $K$ be a totally real $C_2^3$-extension of $\Q$. Then there are wandering units under iteration of the Mahler measure.
\end{prop}

Now we assume that $K/\Q$ is a totally real Galois extension with $\Gal(K/\Q)\cong Q_8$. This is, $\Gal(K/\Q)$ is generated by two elements $\sigma, \tau$, satisfying $\sigma^4=\id$, $\tau^2=\sigma^2$, and $\tau\sigma\tau^{-1}=\sigma^3$. Hence, a full list of embeddings of $K$ is given by $\id,\sigma,\sigma^2,\sigma^3,\tau,\tau^3,\sigma\tau,\sigma\tau^3$. By Dirichlet's unit theorem, there is an algebraic unit in $K$ satisfying
\begin{equation}\label{eq:distributionQ8}
\abs{\alpha}>\abs{\tau(\alpha)}>\abs{\sigma\tau^3(\alpha)}>\abs{\sigma^3(\alpha)}>1>\abs{\sigma\tau(\alpha)}>\abs{\sigma(\alpha)}>\abs{\tau^3(\alpha)}>\abs{\sigma^2(\alpha)},
\end{equation}
and
\begin{equation}\label{eq:alphaQ8}
\abs{\alpha\sigma\tau(\alpha)\sigma(\alpha)\tau^3(\alpha)}>1.
\end{equation}
This latter condition just means that $\alpha$ is much larger than the other conjugates. The absolute values of the Galois conjugates of the Mahler measure of $\alpha$ are
\begin{itemize}
\item $M(\alpha)=\abs{\alpha\tau(\alpha)\sigma\tau^3(\alpha)\sigma^3(\alpha)}$
\item $\abs{\sigma(M(\alpha))}=\abs{\sigma(\alpha)\sigma\tau(\alpha)\tau(\alpha)\alpha}$
\item $\abs{\sigma^2(M(\alpha))}=\abs{\sigma^2(\alpha)\tau^3(\alpha)\sigma\tau(\alpha)\sigma(\alpha)}$
\item $\abs{\sigma^3(M(\alpha))}=\abs{\sigma^3(\alpha)\sigma\tau^3(\alpha)\tau^3(\alpha)\sigma^2(\alpha)}$
\item $\abs{\tau(M(\alpha))}=\abs{\tau(\alpha)\sigma^2(\alpha)\sigma^3(\alpha)\sigma\tau(\alpha)}$
\item $\abs{\tau^3(M(\alpha))}=\abs{\tau^3(\alpha)\alpha\sigma(\alpha)\sigma\tau^3(\alpha)}$
\item $\abs{\sigma\tau(M(\alpha))}=\abs{\sigma\tau(\alpha)\sigma^3(\alpha)\alpha\tau^3(\alpha)}$
\item $\abs{\sigma\tau^3(M(\alpha))}=\abs{\sigma\tau^3(\alpha)\sigma(\alpha)\sigma^2(\alpha)\tau(\alpha)}$
\end{itemize}
By our assumption \eqref{eq:alphaQ8} on $\alpha$, we see that $M(\alpha), \abs{\sigma(M(\alpha))}, \abs{\tau^3(M(\alpha))}, \abs{\sigma\tau(M(\alpha))}>1$. Moreover, since $\alpha$ is an algebraic unit, we get
\begin{align*}
1&=\abs{M(\alpha)}\cdot \abs{\sigma^2(M(\alpha))}=\abs{\sigma(M(\alpha))}\cdot\abs{\sigma^3(M(\alpha))} \\ &=\abs{\tau(M(\alpha))}\cdot\abs{\tau^3(M(\alpha))}=\abs{\sigma\tau(M(\alpha))}\cdot\abs{\sigma\tau^3(M(\alpha))}.
\end{align*}
Hence, using once again that $\alpha$ is an algebraic unit, we get
\begin{align*}
M^{2}(\alpha)&=\abs{M(\alpha)\sigma(M(\alpha))\tau^3(M(\alpha))\sigma\tau(M(\alpha))}\\
 &=\abs{\alpha^4 \tau(\alpha)^2\sigma\tau^3(\alpha)^2\sigma^3(\alpha)^2\sigma\tau(\alpha)^2\tau^3(\alpha)^2\sigma(\alpha)^2} = \abs{\frac{\alpha}{\sigma^2(\alpha)}}^2.
\end{align*}
It follows immediately that the absolute values of the Galois conjugates of $M^{2}(\alpha)$ are
\[
\abs{\frac{\alpha}{\sigma^2(\alpha)}}^2, \abs{\frac{\sigma^2(\alpha)}{\alpha}}^2, \abs{\frac{\sigma(\alpha)}{\sigma^3(\alpha)}}^2, \abs{\frac{\sigma^3(\alpha)}{\sigma(\alpha)}}^2, \abs{\frac{\tau(\alpha)}{\tau^3(\alpha)}}^2, \abs{\frac{\tau^3(\alpha)}{\tau(\alpha)}}^2, \abs{\frac{\sigma\tau(\alpha)}{\sigma\tau^3(\alpha)}}^2, \abs{\frac{\sigma\tau^3(\alpha)}{\sigma\tau(\alpha)}}^2.
\]
By \eqref{eq:distributionQ8} we find
\[
M^{3}(\alpha)=\abs{\frac{\alpha}{\sigma^2(\alpha)}}^2 \abs{\frac{\sigma^3(\alpha)}{\sigma(\alpha)}}^2 \abs{\frac{\tau(\alpha)}{\tau^3(\alpha)}}^2 \abs{\frac{\sigma\tau^3(\alpha)}{\sigma\tau(\alpha)}}^2 = \abs{\frac{M(\alpha)}{\sigma^2(M(\alpha))}}^2.
\]
We have already noted that $\abs{\sigma^2(M(\alpha))}=\abs{M(\alpha)}^{-1}$. Thus, finally we get
\[
M^{3}(\alpha)=M(\alpha)^4.
\]
By Lemma \ref{lem:torsionfree} we can conclude that $\alpha$ is a wandering unit under iteration of $M$. This proves that in every totally real $Q_8$-extension of $\Q$ there are wandering units under iteration of $M$. Combining this with Proposition \ref{prop:3-quad} and our preliminary discussion yields:
\begin{thm}
If $K/\Q$ is a totally real Galois extension of degree $8$, then $K$ contains algebraic units which are wandering under iteration of $M$.
\end{thm}

\section{Galois extensions of degree nine}

Let $K/\Q$ be a Galois extension of degree nine. Then its Galois group is either isomorphic to $C_9$ or $C_3\times C_3$. If it is $C_9$ then $K$ contains a wandering unit by the more general statement Proposition \ref{prop:cyclic} proved in Section \ref{sec:cyclic}. Hence, the next proposition shows that $K$ contains a wandering unit in either case. 

\begin{prop}\label{prop:C3C3}
Let $K/\Q$ be a Galois extension with Galois group isomorphic to $C_3 \times C_3$. Then there are wandering units in $K$.
\end{prop}
\begin{proof}
By Galois theory, we can write $K$ as the compositum of two linearly disjoint fields $K_1$ and $K_2$, both of degree $3$. Since $K$ is Galois and of odd degree, it must be totally real. In particular $K_1$ and $K_2$ are totally real.

Let $\alpha$ be a Pisot unit in $K_1$ and let $\beta$ be a Pisot unit in $K_2$ (recall that any real number field of degree greater than $1$ contains a Pisot unit). Denote by $\alpha_1$, $\alpha_2$, $\alpha_3$ the Galois conjugates of $\alpha$, and by $\beta_1$, $\beta_2$, $\beta_3$ the Galois conjugates of $\beta$, ordered such that
\[
\vert \alpha_1 \vert > \vert \alpha_2\vert > \vert \alpha_3 \vert \qquad \text{ and } \qquad \vert \beta_1 \vert > \vert \beta_2\vert > \vert \beta_3 \vert.
\]
Since $K_1$ and $K_2$ are linearly disjoint, the Galois conjugates of $\alpha_1 \beta_1$ are precisely
\[
\alpha_1 \beta_1, \quad \alpha_1 \beta_2, \quad \alpha_1 \beta_3, \quad \alpha_2 \beta_1, \quad \alpha_2 \beta_2, \quad \alpha_2 \beta_3, \quad \alpha_3 \beta_1, \quad \alpha_3 \beta_2, \quad \alpha_3 \beta_3. 
\]
We may raise $\alpha$ and $\beta$ to some power to assume that $\alpha$ and $\beta$ are sufficiently close together to satisfy 
\[
\vert \alpha_1 \beta_3\vert > 1 \qquad \text{ and } \qquad \vert \beta_1 \alpha_3 \vert >1.
\]
Then we find
\[
M(\alpha_1\beta_1) =\vert (\alpha_1 \beta_1)(\alpha_1 \beta_2)(\alpha_1 \beta_3)(\alpha_2 \beta_1)(\alpha_3 \beta_1)\vert = \vert \alpha_1^2 \beta_1^2 \vert =(\alpha_1\beta_1)^2.
\]
By construction $\alpha_1\beta_1$ is a Perron number, and hence torsion free, and hence $\alpha_1\beta_1\in K$ is a wandering unit under iteration of $M$ by Lemma \ref{lem:torsionfree}.
\end{proof}

\section{Cyclic extensions of odd degree $\geq 5$}\label{sec:cyclic}
Let $n\geq 5$ be an odd integer, and let $K/\QQ$ be a Galois extension with $\Gal(K/\QQ) \cong C_n$. Note that $K$ is necessarily totally real, as a Galois extension of odd degree. For any $\alpha \in K$ and any $\sigma \in \Gal(K/\Q)$ we define $\consalph{i}=\sigma^i(\alpha)$. In particular, we always have $\consalph{0}=\alpha$, and for any $i\in\mathbb{Z}$ we have $\consalph{i}=\consalph{i+n}$.

In order to prove the main result of this section, we need a technical lemma:
\begin{lemma}\label{lem:distribution2}
Let $\sigma$ be some generator of $\Gal(K/\Q)\cong C_n$ and let $\alpha \in K$ be an algebraic unit such that
\begin{equation}\label{eq:alphadistribution2}
\abs{\consalph{0}}>\abs{\consalph{1}}>\abs{\consalph{-1}}>\abs{\consalph{2}}>\abs{\consalph{-2}}>\ldots,
\end{equation} that is,

\begin{enumerate}[label=(\roman*),start=1]
\item $\abs{\consalph{-i}} > \abs{\consalph{1+i}} ~\forall~i\in\{0,\ldots,\frac{n-3}{2}\}$.
\item $\abs{\consalph{i}}>\abs{\consalph{-i}} ~\forall~i\in\{1,\ldots,\frac{n-1}{2}\}$.
\end{enumerate}
Assume that precisely $\frac{n+1}{2}$ conjugates of $\alpha$
lie outside of the unit circle. Then $\beta= M(\alpha)$ has either precisely $\frac{n+1}{2}$ or precisely $\frac{n-1}{2}$ conjugates outside the unit circle. Moreover, we have:
\begin{itemize}
\item If $n\equiv 3 \mod{4}$ the conjugates of $\beta$ are distributed in the following fashion:
\begin{equation}\label{eq:betadistribution2}
\abs{\consbet{0}}>\abs{\consbet{-1}}>\abs{\consbet{1}}>\abs{\consbet{-2}}>\abs{\consbet{2}}>\ldots
\end{equation}
That is,
\begin{enumerate}[label=(\alph*),start=1]
\item $\abs{\consbet{i}}>\abs{\consbet{-1-i}} ~\forall~i\in\{0,\ldots,\frac{n-3}{2}\}$, and 
\item $\abs{\consbet{-i}} > \abs{\consbet{i}} ~\forall~i\in\{1,\ldots,\frac{n-1}{2}\}$.
\end{enumerate}
\item If $n\equiv 1 \mod{4}$ the Galois conjugates of $\beta$ are distributed as in \eqref{eq:alphadistribution2}.
\end{itemize}
\end{lemma}

If we assume for the moment the validity of this lemma, then a converse statement is also true:
\begin{lemma}\label{lem:distribution1}
Let $\sigma$ be a generator of $\Gal(K/\Q)\cong C_n$, and let $\alpha \in K$ be an algebraic unit with Galois conjugates distributed as in \eqref{eq:betadistribution2} with precisely $\frac{n+1}{2}$ conjugates outside the unit circle. Then $\beta=M(\alpha)$ has either precisely $\frac{n+1}{2}$ or precisely $\frac{n-1}{2}$ conjugates outside the unit circle, and
\begin{itemize}
\item if $n\equiv 3 \mod{4}$ the Galois conjugates of $\beta$ are distributed as in \eqref{eq:alphadistribution2}.
\item if $n\equiv 1\mod{4}$ the Galois conjugates of $\beta$ are distributed as in \eqref{eq:betadistribution2}.
\end{itemize}
\end{lemma}
\begin{proof}
This is precisely the statement from Lemma \ref{lem:distribution2} with $\sigma$ replaced by $\sigma^{-1}$. Since Lemma \ref{lem:distribution2} holds true for any generator $\sigma$, both lemmas are equivalent.
\end{proof}

Putting the two lemmas together we obtain the following result:
\begin{prop}\label{prop:cyclic}
Let $n\geq 5$ be an odd integer. Moreover, let $K/\QQ$ be a Galois extension with Galois group isomorphic to the cyclic group $C_n$. Then $K$ contains a wandering unit under iteration of the Mahler measure.
\end{prop}
\begin{proof}
Fix any generator $\sigma\in\Gal(K/\Q)$. We define four sets of algebraic units in $K$. 
\begin{itemize} \item[$A_1$:] The set of algebraic units in $K$ with a distribution of Galois conjugates as in \eqref{eq:alphadistribution2} and precisely $\frac{n+1}{2}$ conjugates outside the unit circle.
\item[$A_2$:]  The set of algebraic units in $K$ with a distribution of Galois conjugates as in \eqref{eq:alphadistribution2} and precisely $\frac{n-1}{2}$ conjugates outside the unit circle.
\item[$B_1$:] The set of algebraic units in $K$ with a distribution of Galois conjugates as in \eqref{eq:betadistribution2} and precisely $\frac{n+1}{2}$ conjugates outside the unit circle.
\item[$B_2$:] The set of algebraic units in $K$ with a distribution of Galois conjugates as in \eqref{eq:betadistribution2} and precisely $\frac{n-1}{2}$ conjugates outside the unit circle.
\end{itemize}
By Lemmas \ref{lem:distribution1} and \ref{lem:distribution2}, we know that for all $\alpha\in A_1 \cup B_1$ we have
\[
M(\alpha) \in A_1 \cup A_2 \cup B_1 \cup B_2.
\]
Let $\alpha \in A_2$, then $(\consalph{-i})^{-1} < (\consalph{1+i})^{-1}$ for all $i\in\{0,\ldots,\frac{n-3}{2}\}$, and $(\consalph{i})^{-1} < (\consalph{-i})^{-1}$ for all $i\in \{1,\ldots,\frac{n-1}{2}\}$. We set $\beta=(\consalph{-\frac{n-1}{2}})^{-1}$. Then precisely $\frac{n+1}{2}$ Galois conjugates of $\beta$ are outside the unit circle, and $(\consalph{i})^{-1}= \consbet{i+\frac{n-1}{2}}$ for all $i\in\Z$. Hence, $\consbet{i+\frac{n-1}{2}}< \consbet{-i+\frac{n-1}{2}}$ for all $i\in\{1,\ldots,\frac{n-1}{2}\}$, and $\consbet{-i+\frac{n-1}{2}}<\consbet{1+i+\frac{n-1}{2}}$ for all $i\in\{0,\ldots,\frac{n-3}{2}\}$. If we substitute $j=-i+\frac{n-1}{2}$, then we see that the Galois conjugates of $\beta$ satisfy the distribution \eqref{eq:betadistribution2}.  Hence, we have $\beta\in B_1$. The same argument proves that the inverse of some Galois conjugate of $\alpha \in B_2$, let's call it $\beta$, is in $A_1$. Since the Mahler measure is Galois-invariant and invariant under taking inverses, it follows that for all $\alpha \in A_2 \cup B_2$ we have $M(\alpha)=M(\beta)\in A_1 \cup A_2 \cup B_1 \cup B_2$.

We have seen that the orbit of any element $\alpha \in A_1$ is contained in $A_1\cup A_2 \cup B_1 \cup B_2$. But obviously none of the elements in this union is a fixed point under $M$. Hence, any element in $A_1$ is a wandering point under iteration of $M$. This proves the proposition, since by Lemma \ref{lem:deg5} the set $A_1$ is non-empty.
\end{proof}

It remains for us to prove Lemma \ref{lem:distribution2}. We break the proof into two cases, where we distinguish between $n\equiv 3 \mod{4}$ and $n\equiv 1 \mod{4}$.

\subsection{Proof of Lemma \ref{lem:distribution2} for $n\equiv 3 \mod{4}$}

Since $\frac{n+1}{2}$ is even, the Galois conjugates of $\alpha$ that lie outside the unit circle are $\consalph{0},\consalph{1},\consalph{-1},\ldots,\consalph{\frac{n+1}{4}}$. Hence, we have
 \[
 \beta=M(\alpha)=\prod_{k=-\frac{n-3}{4}}^{\frac{n+1}{4}} \consalph{k}.
 \]
 This implies $\consbet{i}=\prod_{k=-\frac{n-3}{4}}^{\frac{n+1}{4}} \consalph{k+i}$ for all $i\in\Z$, and in particular
 \begin{equation}\label{eq:betavsalpha}
\frac{\consbet{i}}{\consbet{i-1}}=\frac{\consalph{i+\frac{n+1}{4}}}{\consalph{i-\frac{n+1}{4}}} \quad \forall ~ i\in\Z.
\end{equation}
Define for all $i\in\{0,\ldots,\frac{n-1}{2}\}$,
 \[
 \gamma_i = \frac{\abs{\consbet{i}}}{\abs{\consbet{-i}}}.
 \]
 Notice that we have $\gamma_0=1$. By the definition of $\gamma_i$ and \eqref{eq:betavsalpha} we have for all $i\in\{1,\ldots,\frac{n-1}{2}\}$
 \begin{align}
 \gamma_i &= \gamma_{i-1}\cdot \frac{\abs{\consbet{i}\consbet{-i+1}}}{\abs{\consbet{-i}\consbet{i-1}}} = \gamma_{i-1} \cdot \frac{\abs{\consalph{i+\frac{n+1}{4}} \consalph{-i+\frac{n+5}{4}} }}{\abs{ \consalph{i-\frac{n+1}{4}} \consalph{-i-\frac{n-3}{4}}   }} \nonumber \\
 &= \gamma_{i-1} \cdot \frac{ \abs{ \consalph{ (i+\frac{n-3}{4}) +1 } \consalph{ (-i+\frac{n+1}{4})+1 } } }{\abs{ \consalph{-(i+\frac{n-3}{4})} \consalph{-(-i+\frac{n+1}{4})}  } }\label{eq:gamma1}\\
 &= \gamma_{i-1} \cdot \frac{\abs{\consalph{-(i-\frac{n+5}{4})} \consalph{-(-i+\frac{3n-1}{4})} }}{\abs{ \consalph{(i-\frac{n+5}{4})+1} \consalph{(-i+\frac{3n-1}{4})+1}   }} \label{eq:gamma2}.
 \end{align}
For the latter equation we have used that we can shift the indicies by $n$. Using equation \eqref{eq:gamma1} and assumption (i) on the $\consalph{i}$'s we see that 
\[
1=\gamma_0>\gamma_1> \ldots > \gamma_{\frac{n-3}{4}}.
\]
Using instead equation \ref{eq:gamma2} and assumption (i) gives us
\[
\gamma_{\frac{n+1}{4}} < \gamma_{\frac{n+5}{4}}< \ldots < \gamma_{\frac{n-1}{2}}.
\]
We use again that the indicies are all understood to be modulo $n$. Then we have
\[
\gamma_{\frac{n-1}{2}}=\frac{\abs{\consbet{\frac{n+1}{2}+1}}}{\abs{\consbet{\frac{n+1}{2}}}}\overset{\eqref{eq:betavsalpha}}{=} \frac{\abs{\consalph{\frac{n+1}{4}}}}{\abs{\consalph{-\frac{n+1}{4}}}}\overset{(i)}{<}1.
\]
We have just seen that $\gamma_i < 1$ for all $i\in\{1,\ldots,\frac{n-1}{2}\}$. This means
\[
\abs{\consbet{i}}<\abs{\consbet{-i}} \quad \forall\ i\in \bigg\{1,\ldots,\frac{n-1}{2}\bigg\},
 \]
 which proves the inequalities (b) from the lemma. Now we prove that the $\consbet{i}$'s satisfy the condition (a). Define for all $i\in\{1,\ldots,\frac{n-1}{2}\}$,
 \[
 \mu_i = \frac{\abs{\consbet{i}}}{\abs{\consbet{-1-i}}}.
 \]
 Note that $\mu_0>1$, since $\beta=M(\alpha)$ is a Perron number. As before we calculate

 \begin{align*}
 \mu_{i} &= \mu_{i-1}\cdot \frac{ \abs{\consalph{i+\frac{n+1}{4}} \consalph{-i+\frac{n+1}{4}} } }{\abs{ \consalph{i-1-\frac{n-3}{4}} \consalph{-1-i-\frac{n-3}{4}}  }}\\ &= \mu_{i-1}\cdot \frac{ \abs{\consalph{i+\frac{n+1}{4}} \consalph{-i+\frac{n+1}{4}} } }{\abs{  \consalph{-(i+\frac{n+1}{4})} \consalph{-(-i+\frac{n+1}{4})}}}
 \end{align*}
 for all $i\in\{1,\ldots,\frac{n-1}{2}\}$. By assumption (ii) on the $\consalph{j}$'s, we see 
 \[
 1<\mu_0<\mu_1<\mu_2<\ldots<\mu_{\frac{n-3}{4}}.
 \]
 and
 \[
 \mu_{\frac{n+1}{4}}>\mu_{\frac{n+9}{4}}>\ldots > \mu_{\frac{n-1}{2}}.
 \]
Moreover, we have $\mu_{\frac{n-1}{2}}=\frac{\abs{\consbet{\frac{n-1}{2}}}}{\abs{\consbet{1-\frac{n-1}{2}}}}=\frac{\abs{\consbet{\frac{n-1}{2}}}}{\abs{\consbet{\frac{n-1}{2}}}}=1$. We have just seen that $\mu_i >1$ for all $i\in\{0,\ldots,\frac{n-3}{2}\}$. This is equivalent to condition (a) of the lemma.

The distribution of the conjugates of $\beta$ (which we just proved) implies
  \[
 \underbrace{\ldots}_{\frac{n-3}{2}\text{-roots}} > \abs{\consbet{\frac{n-3}{4}}} > \abs{\consbet{-\frac{n+1}{4}}} > \abs{\consbet{\frac{n+1}{4}}} >  \underbrace{\ldots}_{\frac{n-3}{2}\text{-roots}}
  \]
Hence, we need to show that we have $\abs{\consbet{\frac{n-3}{4}}}>1$ and $\abs{\consbet{\frac{n+1}{4}}}<1$. 

We have
\begin{equation}\label{eqn:beta-sig-n34}
\consbet{\frac{n-3}{4}}=\prod_{k=-\frac{n-3}{4}}^{\frac{n+1}{4}}\consalph{k+\frac{n-3}{4}} = \prod_{k=0}^{\frac{n-1}{2}}\consalph{k}=\consalph{0}\cdot \prod_{k=1}^{\frac{n-1}{2}}\consalph{k}.
\end{equation}
Assume for the sake of contradiction that $\abs{\consbet{\frac{n-3}{4}}}\leq 1$, then by \eqref{eqn:beta-sig-n34} and (ii) we also have \[\bigg\lvert\consalph{0}\cdot \prod_{k=1}^{\frac{n-1}{2}}\consalph{-k}\bigg\rvert <1.\] Since $\alpha$ is an algebraic unit, it has norm $\pm 1$. Hence multiplying the two terms leads to
\[
1 > \abs{\consalph{0}}\cdot \bigg\lvert\prod_{k=-\frac{n-1}{2}}^{\frac{n-1}{2}}\consalph{k}\bigg\rvert = \abs{\consalph{0}}. 
\]
This is a contradiction, and hence $\abs{\consbet{\frac{n-3}{4}}}>1$. Lastly we need to prove $\abs{\consbet{\frac{n+1}{4}}}<1$. We argue similarly: Assume for the sake of contradiction that
\[
1\leq \abs{\consbet{\frac{n+1}{4}}}= \prod_{k=-\frac{n-3}{4}}^{\frac{n+1}{4}}\abs{\consalph{k+\frac{n+1}{4}}}=\abs{\consalph{\frac{n+1}{2}}}\cdot \bigg\lvert\prod_{k=1}^{\frac{n-1}{2}}\consalph{k}\bigg\rvert.
\]
Then, by (i), we also have that $1< \abs{\consalph{\frac{n+1}{2}}\cdot \prod_{k=0}^{\frac{n-3}{2}}\consalph{-k}} = \abs{\prod_{k=0}^{\frac{n-1}{2}}\consalph{-k}}$. We multiply these two terms to conclude that
\[
1< \abs{\consalph{\frac{n+1}{2}}}\cdot \bigg\lvert\prod_{k=-\frac{n-1}{2}}^{\frac{n-1}{2}}\consalph{k}\bigg\rvert=\abs{\consalph{\frac{n+1}{2}}}=\abs{\consalph{-\frac{n-1}{2}}}.
\]
This is a contradiction, and hence $1>\abs{\consbet{\frac{n+1}{4}}}$, which concludes the proof of Lemma \ref{lem:distribution2} in the case $n\equiv 3 \mod{4}$.

\subsection{Proof of Lemma \ref{lem:distribution2} for $n\equiv 1 \mod{4}$}

The proof is essentially the same as the proof in the case $n\equiv 3 \mod{4}$. We will only present the main differences. Set $\beta=M(\alpha)$. Then
\[
\consbet{i}=\prod_{k=-\frac{n-1}{4}}^{\frac{n-1}{4}} \consalph{i+k} \quad \forall ~ i\in\Z.
\]
Again we define $\gamma_i = \frac{\abs{\consbet{i}}}{\abs{\consbet{-i}}}$. Then, we have
\[
\frac{\gamma_i}{\gamma_{i-1}} = \frac{\abs{ \consalph{i+\frac{n-1}{4}} \consalph{-i+\frac{n+3}{4}} }}{ \abs{ \consalph{-(i+\frac{n-1}{4})} \consalph{-(-i+\frac{n+3}{4})} }  }.
\]
By (ii) this implies
\[
1=\gamma_0< \gamma_1 < \ldots < \gamma_{\frac{n-1}{4}},
\]
and
\[
\gamma_{\frac{n+3}{4}}> \gamma_{\frac{n+7}{4}}> \ldots > \gamma_{\frac{n-1}{2}}=1.
\]
We have just seen that $\gamma_i >1$ for all $i\in\{1,\ldots,\frac{n-1}{2}\}$, which proves that the Galois conjugates of $\beta$ satisfy the inequalities (ii). To prove the inequalities (i), we set $\mu_i=\frac{\abs{\consbet{-i}}}{\abs{\consbet{1+i}}}$. Then,
\[
\mu_i = \mu_{i-1} \cdot \frac{\abs{ \consalph{-i-\frac{n-1}{4}} \consalph{i-\frac{n-1}{4}} }}{\abs{ \consalph{-i+\frac{n+3}{4}} \consalph{i+\frac{n+3}{4}} }}.
\]
By (i), we conclude
\[
1<\mu_0 < \mu_1 < \ldots < \mu_{\frac{n-5}{4}},
\]
and
\[
\mu_{\frac{n-1}{4}}>\mu_{\frac{n+3}{4}} > \ldots > \mu_{\frac{n-1}{2}}>1.
\]
Hence, it is $\mu_1 > 1$ for all $i\in\{0,\ldots,\frac{n-3}{2}\}$, which proves the inequalities (i) for the conjugates of $\beta$. 

We are left to prove that either precisely $\frac{n-1}{2}$ or precisely $\frac{n+1}{2}$ conjugates of $\beta$ are outside the unit circle. Since we know that $\beta$ satisfies \eqref{eq:alphadistribution2}, we have
\[
\underbrace{\ldots}_{\frac{n-3}{2}\text{-roots}} > \abs{\consbet{\frac{n-1}{4}}} > \abs{\consbet{-\frac{n-1}{4}}} > \abs{\consbet{\frac{n+3}{4}}} >  \underbrace{\ldots}_{\frac{n-3}{2}\text{-roots}}
\]
and we need to show $\abs{\consbet{\frac{n-1}{4}}}>1$ and $\abs{\consbet{\frac{n+3}{4}}} <1$. Assume for the sake of contradiction that 
\[
1\geq \abs{\consbet{\frac{n-1}{4}}}=\abs{\consbet{0}}\cdot \bigg\lvert\prod_{k=1}^{\frac{n-1}{2}} \consbet{k}\bigg\rvert \overset{(ii)}{>} \abs{\consbet{0}}\cdot \bigg\lvert\prod_{k=1}^{\frac{n-1}{2}} \consbet{-k}\bigg\rvert.
\]
Then we also would have $1>\abs{\consbet{0}}\cdot \abs{\prod_{k=-\frac{n-1}{2}}^{\frac{n-1}{2}} \consbet{k}}=\abs{\consbet{0}}$, which gives a contradiction. Lastly assume for the sake of contradiction that
\[
1\leq\abs{\consbet{\frac{n+3}{4}}} = \abs{\consbet{-\frac{n-1}{2}}}\cdot \bigg\lvert\prod_{k=1}^{\frac{n-1}{2}} \consbet{k}\bigg\rvert \overset{(i)}{<} \abs{\consbet{-\frac{n-1}{2}}}\cdot \bigg\lvert\prod_{k=0}^{\frac{n-3}{2}} \consbet{-k}\bigg\rvert.
\]
Then also $1< \abs{\consbet{-\frac{n-1}{2}}}\cdot \abs{\prod_{k=-\frac{n-1}{2}}^{\frac{n-1}{2}} \consbet{k}}=\abs{\consbet{-\frac{n-1}{2}}}$, which is a nonsense by \eqref{eq:alphadistribution2}. Hence, $1>\abs{\consbet{\frac{n+3}{4}}}$, proving Lemma \ref{lem:distribution2} and thus Proposition \ref{prop:cyclic}.

\section{Proof of Theorem \ref{thm:main-abelian}}

We wish to classify all abelian extension $K/\Q$ such that all elements in $K$ are preperiodic under iteration of $M$. Let $K$ be an arbitrary abelian number field, and denote by $K^+$ the maximal real subfield of $K$. Then $K^+$ is necessarily totally real and $M(\alpha)\in K^+$ for all $\alpha \in K$. In particular, all elements of $K$ are preperiodic under iteration of $M$, if and only if all elements of $K^+$ are preperiodic.

We know that there are no wandering points in degree $\leq 3$. In addition with Proposition \ref{prop:totreal4}, we see that there are no wandering points in $K^+$ if $\Gal(K^+/\Q)$ is isomorphic to $C_1$, $C_2$, $C_3$, or $C_2\times C_2$. If $\Gal(K^+/\Q)$ is not isomorphic to one of these groups, then one of the following groups appear as a quotient group of $\Gal(K^+/\Q)$:
\begin{enumerate}[label=(\alph*),start=1]
\item $C_4$
\item $C_2 \times C_2 \times C_2$
\item $C_n$, where $n\geq 5$ is odd
\item $C_6$
\item $C_3 \times C_3$
\end{enumerate}
This is, $K^+$ contains a subfield $F$ such that $\Gal(F/\Q)$ is isomorphic to some group from the list above. But in any of these cases, we know that there are wandering units in $F\subseteq K$ (see Propositions \ref{prop:totreal4}, \ref{prop:3-quad}, \ref{prop:cyclic}, \ref{prop:Gal6}, \ref{prop:C3C3}). This proves Theorem \ref{thm:main-abelian}.

\bibliographystyle{abbrv} 
\bibliography{bib}

\end{document}